\documentclass[11pt,reqno]{amsart}
\usepackage{amssymb,amsmath}
\usepackage{amsmath,amstext,amssymb,amscd}
\usepackage{verbatim}
\usepackage{enumerate}
\usepackage{mathrsfs}
\usepackage[dvipsnames,usenames]{xcolor}

\usepackage{amsthm}

\usepackage{color}
\def\Bbb{\mathbb}
\def\Cal{\mathcal}
\def\Dt{\partial_t}

\def\eb{\varepsilon}

\def\R {\mathbb{R}}
\def\N {\mathbb{N}}

\def\Pr{\operatorname{Pr}}

\def\E {{\mathcal E}}

\def\<{\left<}
\def\>{\right>}

\def\Nx{\nabla_x}
\def\Dx{\Delta_x}
\def\({\left(}
\def\){\right)}
\def\divv{\operatorname{div}}

\newtheorem{proposition}{Proposition}[section]
\newtheorem{theorem}[proposition]{Theorem}
\newtheorem{corollary}[proposition]{Corollary}
\newtheorem{lemma}[proposition]{Lemma}

\theoremstyle{definition}
\newtheorem{definition}[proposition]{Definition}

\newtheorem{remark}[proposition]{Remark}
\newtheorem{example}[proposition]{Example}

\numberwithin{equation}{section}

\def \no#1#2#3 {{\bf #1} (#3), #2.}
\def \eds#1#2#3 {#1, #2, #3.}

\title[]
{Entropy estimates for  uniform attractors of dissipative PDEs with non translation-compact external forces}
\author[Y. Xiong, A. Kostianko,  C. Sun and S. Zelik] {Yangmin Xiong${}^{1}$, Anna Kostianko${}^{1,2}$, Chunyou Sun${}^1$, and Sergey Zelik${}^{1,3,4}$}
\begin{document}
\begin{abstract} We study the Kolmogorov's entropy of uniform attractors for non-autonomous dissipative PDEs. The main attention is payed to the case where the external forces are not translation-compact. We present a new general scheme which allows us to give the upper bounds of this entropy for various classes of external forces through the entropy of proper projections of their hulls to the space of translation-compact functions. This result generalizes well known estimates of Vishik and Chepyzhov for the translation-compact case. The obtained results are applied to three model problems: sub-quintic 3D damped wave equation with Dirichlet boundary conditions,   quintic 3D wave equation with periodic boundary conditions and 2D Navier-Stokes system in a bounded domain. The examples of finite-dimensional uniform attractors for some special external forces which are not translation-compact are also given.
\end{abstract}

\address{${}^1$ School of Mathematics and Statistics, Lanzhou University, Lanzhou  \\ 730000,
P.R. China}
\email{xiongym18@lzu.edu.cn (Y.Xiong), sunchy@lzu.edu.cn (C.Sun)}
\address{${}^2$ Imperial College, London SW7 2AZ, United Kingdom.}
\email{a.kostianko@imperial.ac.uk (A.Kostianko)}
\address{${}^3$ University of Surrey, Department of Mathematics, Guildford, GU2 7XH, United Kingdom.}
\email{s.zelik@surrey.ac.uk (S.Zelik)}

\address{${}^4$ Keldysh Institute of Applied Mathematics, Moscow, Russia}

\thanks{  This work is partially supported by the grant 19-71-30004 of RSF (Russia),
  the Leverhulme grant No. RPG-2021-072 (United Kingdom) and NSFC grant No. 11522109 and 11871169 (China).}
\maketitle

\tableofcontents
\section{Introduction}\label{sec1}
It is well-known that in many cases the long-time behaviour of
dissipative partial differential equations (PDEs) can be described in terms of the so-called
global attractors. Being a compact invariant subset of a phase space attracting
the images of all bounded sets when time tends to infinity, a global attractor contains
all the non-trivial dynamics of the system considered. On the other hand, it is
usually essentially smaller than the initial phase space. In particular, in many
cases this attractor has finite Hausdorff and fractal dimensions, so despite the
infinite-dimensionality of the initial phase space (e.g., $E= L^2(\Omega)$), the limit-reduced
dynamics on the attractor is in a sense finite-dimensional and can be described by
finitely many parameters, see \cite{BV, Temam, CV} and references therein.
\par
The situation becomes more complicated when the considered dissipative PDE
is non-auto\-no\-mous, for instance, contains the external forces depending explicitly
on time, i.e., the underlying PDE may have the form
\begin{equation}\label{1.1}
\Dt u = A(u) + g(t), \  u\big|_{t=\tau}=u_\tau\in E,
\end{equation}
where $A(u)$ is some non-linear operator which we will not specify here, see the
examples in Sections \ref{s4}, \ref{s.5}, \ref{s6} below, $E$ is a phase space of the problem considered (we
will assume below that $E$ is a reflexive Banach space) and the external forces $g$
are assumed to belong to the space $L^p_b (\R; H )$ with $1 < p < \infty$, where $H$ is another
reflexive Banach space or the space $M_b(\R,H)$ of $H$-valued measures, see Section \ref{s.2n} for more detailed exposition.
\par
At present time there exist two major ways to extend the concept of a global
attractor to the case of non-autonomous PDEs. The first one treats the attractor
of a non-autonomous system as a time-dependent set as well $\Cal A = \Cal A(t)\subset E$, $t\in\R$.
This naturally leads to the so-called pullback attractors or the kernel sections in the
terminology of Vishik and Chepyzhov, see \cite{4, CV, 8,9} and references therein. One of the
main advantages of this approach is the fact that the attractor $\Cal A(t)$ usually remains
finite-dimensional for every $t$ and it is also well adapted to study random/stochastic
PDEs, see \cite{9}. However, in the case of deterministic PDEs, this approach has an
essential drawback, namely, the attraction forward in time is usually lost and we
have only a weaker form of the attraction property pullback in time. As a result,
an exponentially repelling forward in time trajectory may be a pullback attractor,
see \cite{MirZel} and references therein. We also mention that this problem can be overcome
using the concept of the so-called non-autonomous exponential attractor, see \cite{EMZPRSE}.
 \par
  The alternative approach is based on the reduction of the non-autonomous
dynamical system (DS) to the autonomous one and  treats the attractor of a nonautonomous DS as a time-independent set $\Cal A \subset E$. This approach naturally leads
to the so-called uniform attractor which is the main object of investigation of the
present paper, so we explain it in a bit more detailed way. Following the general
scheme, one should consider the family of equations of the form \eqref{1.1}:
\begin{equation}\label{1.2}
\Dt u = A(u) + h(t),\ \  u\big|_{t=\tau}=u_\tau,\  h\in\Cal H(g)
\end{equation}
with all external forces belonging to the so-called hull $\Cal H(g)$ of the initial external
force $g$ generated by all time shifts of the initial external force $g$ and their closure
in the properly chosen topology, see Sections \ref{s.2n} and \ref{s3.m} for more details.
\par
Then, assuming that the problems \eqref{1.2} are globally well-posed in $E$, we have
a family of dynamical processes $U_h(t, \tau ) : E \to E$, $h\in \Cal H(g)$, in the phase space
$E$ generated by the solution operators $U_h(t,\tau )u_\tau:= u(t)$ of \eqref{1.2}. Introduce an
extended phase space $\Bbb E := E\times\Cal H(g)$ associated with problem \eqref{1.2}. Then, the
extended semigroup on $\Bbb E$ is defined as follows:
\begin{equation}\label{1.3}
\Bbb S(t)(u_0, h) := (U_h(t,0)u_0, T(t)h),\ \  u_0\in  E,\ \  h\in  \Cal H(g),\ \  (T(s)h)(t) := h(t+s).
\end{equation}
Finally, if this semigroup possesses a global attractor $\Bbb A\subset\Bbb E$, then its projection
$\Cal A := \Pi_1\Bbb A$ is called a uniform attractor associated with the family of equations
\eqref{1.2}, see \cite{5, CV, BV, 7} as well as Section \ref{s3.m} for more details.
\par
The choice of the topology on the extended phase space is crucial for
this approach. According to Vishik and Chepyzhov, see \cite{CV}, there are two natural
choices of this topology. The first one is the {\it weak} topology which leads to the
so-called weak uniform attractor. In this case, we take the topology induced by the
embedding $\Cal H(g)\subset L^p_{loc,w}(\R; H )$ on the hull $\Cal H(g)$ (this gives the compactness of the hull since the space $L^p_{loc}(\R; H )$ is reflexive and bounded sets are precompact in
a weak topology, see e.g. \cite{22}) or the local weak star topology if the space of measures $M_b(\R,H)$ is considered  combined with  the {\it weak} topology on the phase space $E$. In this case
we do not need extra assumptions on the external forces $g$ and only the translation
boundedness: $g\in L^p_b (\R; H)$ (resp. $g\in M_b(\R,H)$) is usually sufficient to have a weak uniform attractor $\Cal A$.
\par
The second natural choice is the choice of {\it strong} topologies on both components
$E$ and $\Cal H(g)$. In this case, we need the extra assumption that the hull $\Cal H(g)$ is compact
in a {\it strong} topology of $L^p_{loc}(\R; H )$ or $M_{loc}(\R,H)$ (the external forces satisfying this condition are usually referred as translation-compact). Thus, this alternative choice requires
the extra assumption for the external forces to be translation-compact and gives
the so-called strong uniform attractor $\Cal A\subset E$, see \cite{5,CV} for many applications of
this theory for various equations of mathematical physics.
\par
However, as has been pointed out later, there is one more a bit surprising choice
of the topologies when one takes the {\it strong} topology on the $E$-component of the
phase space $\Bbb E$ and the {\it weak or weak star} topology on the hull $\Cal H(g)$. Then, it is possible in many
cases to verify the existence of a strong uniform attractor $\Cal A\subset E$ for the case when
the external forces $g$ are not translation compact. Of course, to gain this strong
compactness, we need some extra assumptions on $g$, but these assumptions can be
essentially weaker than the translation compactness. For instance, in
the case of parabolic PDEs in bounded domains, one can consider the so-called normal or weakly normal external forces,  see  \cite{1, LWZ, 13,14,15} and Section \ref{s.2n}
for more details. These classes of external forces are not sufficient for hyperbolic equations and should be replaced, e.g. by space or time regular ones, see \cite{ZelDCDS,16,17} and references therein for more details. Note also that it looks natural to consider external forces which are Borel measures in time  ($g\in M_b(\R,H)$) for the case of damped hyperbolic equations, see \cite{SZUMN} and Section \ref{s.2n} for more details.
 \par
 We emphasize that, in contrast to the autonomous case or the case of pullback attractors, the fractal dimension of a uniform attractor is typically infinite for more or less general time-dependent external forces (although there are important exceptions, see \cite{CV,7,MirZel} and Section \ref{s7} below), so it looks natural (following Vishik and Chepyzhov \cite{CV}) to use Kolmogorov's $\eb$-entropy in order to control the size of such attractors. Recall that, by definition, for any $\eb>0$, the Kolmogorov's $\eb$-entropy of a compact set $\Cal A$ in a metric space $E$ is the logarithm of the minimal number $N_\eb(\Cal A,E)$ of $\eb$-balls in $E$ which is sufficient to cover the set $\Cal A$:
 $$
 \Bbb H_\eb(\Cal A,E):=\log_2N_\eb(\Cal A,E).
 $$
 Roughly speaking, we have $\Bbb H_\eb(\Cal A,E)\sim d_f(\Cal A)\log_2\frac1\eb$ as $\eb\to0$ if the set $\Cal A$ has a finite fractal dimension $d_f(\Cal A)$ and the asymptotic is just different in the case where the attractor is infinite-dimensional. Thus, the problem of finding this asymptotic in terms of the physical parameters of the system considered and the properties of the external forces $g$ arises.
\par
This problem looks completely understood in the case where the external forces are trans\-la\-tion-compact. Indeed, in this case, an optimal estimate in the form
\begin{equation}\label{0.ennttrr}
\Bbb H_\eb(\Cal A,E)\le C\log_2\frac{R_0}\eb+\Bbb H_{\frac{\eb}K}\(\Cal H(g)\big|_{[0,L\log_2\frac{R_0}\eb]}, L^p_b([0,L\log_2\frac{R_0}\eb],H)\)
\end{equation}
is available under some natural assumptions on the system considered. Here the constants $C$, $K$, $R_0$ and $L$ are independent of $\eb\to0$, see \cite{CV} for more details (see also \cite{MirZel,ZelNach,ZeDCDS} and references therein for the extension of this universal formula to the case of locally compact attractors which correspond to dissipative PDEs in unbounded domains).
\par
The main aim of the present paper is to extend entropy estimate \eqref{0.ennttrr} to the case where the external forces are not translation-compact. Note that the second term in the right-hand side is infinite in this case, so \eqref{0.ennttrr} makes no sense and must be corrected. We suggest to modify this universal estimate as follows:
\begin{equation}\label{0.mennttrr}
\Bbb H_\eb(\Cal A,E)\le C\log_2\frac{R_0}\eb+\Bbb H_{\frac{\eb}K}\(\Pr(\eb)\Cal H(g)\big|_{[0,L\log_2\frac{R_0}\eb]}, L^p_b([0,L\log_2\frac{R_0}\eb],H)\),
\end{equation}
where the family of projections $\Pr(\eb):L^p_b(\R,H)\to L^p_{tr-c}(\R,H)$ commutes with the group $T(h)$ of time shifts and makes the projected hull $\Pr(\eb)\Cal H(g)$ translation-compact, see Section 3 for more details. For a general theory we do not require the "projectors" $\Pr(\eb)$ even to be linear, but in applications they are just a composition mollifying operators in space and time whose choice depend on the class of external forces considered. Note also that in the translation-compact case, we may take $\Pr(\eb)=\operatorname{Id}$ and return to the universal formula \eqref{0.ennttrr}. We also mention that an alternative approach to study the Kolmogorov's entropy  of uniform attractors  which is based on the proper metrization of  the weak topology on the hull of the external  forces has been suggested in a very recent paper \cite{XS}.
\par
The paper is organized as follows.
\par
 In Section \ref{s.2n}, we briefly discuss various classes of non translation-compact external forces which give the compactness of the corresponding uniform attractor in a strong topology. An essential attention is payed to spaces of $H$-valued Borel measures of locally finite total variations and their uniformly local analogues $g\in M_b(\R,H)$ as well as their important subspaces. In particular, we introduce here a new class of the so-called {\it weakly regular} external forces/measures which  includes both space and time regular external forces and still gives the strong compactness of attractors of damped wave equations, see Definition \ref{Def0.wreg}.
\par
In Section \ref{s3.m}, we present an abstract scheme of estimating the Kolmogorov's $\eb$-entropy for dissipative PDEs with non translation compact external forces. In particular, we state here the sufficient conditions for the non-autonomous dynamical system considered which guarantee the validity of the modified entropy estimate \eqref{0.mennttrr}. These conditions are traditionally formulated in the form of some asymptotic smoothing properties for the differences of two trajectories of the dissipative PDE considered, see Theorem \ref{Th1.main-ent} and estimate \eqref{1.squeez} below.
\par
In Section \ref{s4}, we apply the obtained result to the damped wave equation
$$
\Dt^2u+\gamma\Dt u-\Dx u+u+f(u)=g(t)
$$
in a bounded smooth domain $\Omega$ of $\R^3$ endowed with Dirichlet boundary conditions and with the non-linearity of sub-quintic growth rate. We proved that under some natural extra assumptions on the non-linearity $f$ and under the assumption that $g$ is a weakly regular $L^2(\Omega)$-valued measure, the uniform attractor $\Cal A$ of the above damped wave equation possesses a modified entropy estimate \eqref{0.mennttrr} in the standard energy phase space, see Theorem \ref{Th2.main-ent} below.
\par
In Section \ref{s.5}, we extend the result of Section \ref{s4} to the nonlinearity $f$ of critical quintic growth rate and periodic boundary conditions. Note that the choice of periodic boundary conditions is somehow unavoidable here since the key energy-to-Strichartz estimate which is crucial for the dissipativity 
is still  unknown for the case of Dirichlet boundary conditions. By this reason, even the existence of a global attractor in the critical case and Dirichlet boundary conditions is known for the autonomous case only, see \cite{11,SZUMN,MSSZ} and also Section \ref{s.5} below for more details. Note also that the results of these two sections look new even for the case of translation-compact external forces. Indeed, the analogous result for the translation compact case has been proved in \cite{CV} for the case of nonlinearities $f$ of cubic and sub-cubic growth rates only.
\par
In Section \ref{s6}, we apply the obtained results to another classical example of a dissipative PDE, namely, for 2D Navier-Stokes equations in a bounded domain endowed with Dirichlet boundary conditions. Here we assume that the external forces $g\in L^2_b(\R,H^{-1}_\sigma(\Omega))$, where the symbol $"\sigma"$ stands for the solenoidal vector fields, and prove that in the case where these external forces are weakly normal, the corresponding uniform attractor $\Cal A$  possesses a modified universal entropy estimate \eqref{0.mennttrr} in the standard energy phase space.
\par
Finally, in Section \ref{s7}, we present two examples of non translation-compact external forces for which the modified universal entropy estimate gives the finiteness of the fractal dimension of the corresponding uniform attractor. We also discuss some related open problems here.

\section{Preliminaries. Classes of external forces}\label{s.2n}
In this section, we briefly recall the known facts about admissible time-dependent external forces and state some results which will be used in the sequel, see \cite{CV,ZelDCDS} for more detailed exposition.
\par
Let $H$ be a reflexive Banach space and let us consider functions $g:\R\to H$.
\begin{definition}\label{Def0.tr-b} Let $1\le p\le\infty$. Then the function $g\in L^p_{loc}(\R,H)$ is  translation bounded if
$$
\|g\|_{L^p_b(\R,H)}:=\sup_{t\in\R}\|g\|_{L^p(t,t+1,H)}<\infty.
$$
The subspace of translation-bounded functions in $L^p_{loc}$ is denoted by $L^p_b(\R,H)$. The Sobolev spaces $W^{s,p}_b(\R,H)$ are defined analogously.
\end{definition}
As known, see e.g. \cite{CV}, it is natural for the attractor theory of non-autonomous equations to consider not only a single external force $g$, but also a family of external forces generated by all time shifts of $g$ as well as their limits in a proper topology. This leads to the following definition.
\begin{definition}\label{Def0.hull} Let $1<p<\infty$ and $g\in L^p_b(\R,H)$. Then the (weak) hull $\Cal H(g)$ is defined by
\begin{equation}\label{0.hull}
\Cal H(g):=[T(h)g]_{L^{p,w}_{loc}(\R,H)}, \ \ (T(h)g)(t):=g(t+h),\ \ t,h\in\R.
\end{equation}
Here and below $[\cdot]_V$ stands for the closure in the topology of $V$ and the symbol "$w$" stands for the weak topology.
\end{definition}
Note that due to the Banach-Alaoglu theorem we know that $\Cal H(g)$ is {\it compact} in $L^{p,w}_{loc}(\R,H)$. It is also not difficult to see that
$$
\|\nu\|_{L^p_b(\R,H)}\le\|g\|_{L^p_b(\R,H)},\ \ \forall\nu\in\Cal H(g).
$$
Since the compactness is crucial for the attractors theory, we will always endow the hull $\Cal H(g)$
by this weak topology.
\par
Let us now consider the limit cases $p=\infty$ and $p=1$. In these cases, $L^p(t,t+1,H)$ is no more reflexive, but using that $L^\infty(t,t+1,H)=(L^1(t,t+1,H^*))^*$ (the space $H$ is assumed reflexive), we are still able to restore the compactness simply by replacing the weak topology by the weak star one. However, this does not work in the case $p=1$ (which is typical for non-autonomous wave equations) and we have to consider {\it measure-valued} external forces in order to restore the compactness.
\par
Namely, following \cite{SZUMN}, we introduce the space of $H$-valued Borel measures $M_{loc}(\R,H)$ with locally finite variations and the space of translation-bounded measures $M_b(\R,H)$  analogously to Definition \ref{Def0.tr-b}, i.e.
$$
\|\mu\|_{M_b(\R,H)}:=\sup_{t\in\R}\|\mu\|_{M(t,t+1;H)}.
$$
We also recall that the norm on $M(t,t+1;H)$ may be defined as follows:
$$
\|\mu\|_{M(t,t+1;H)}=\sup_{\psi\in C(t,t+1;H^*)}\frac{|\int_{[t,t+1]}(\psi(s),\mu(ds))|}{\|\psi\|_{C(t,t+1;H^*)}}.
$$
Moreover, using the duality $M(t,t+1,H)=(C(t,t+1,H^*))^*$ again, we may endow the uniformly local space $M_b(\R,H)$ with the weak-star local topology of $M_{loc}^{w^*}(\R,H)$ and in this space we will have the Banach-Alaoglu theorem. This allows us to define the hull $\Cal H(\mu)$ of the initial measure $\mu\in M_b(\R,H)$ using the weak-star topology and endow it with the topology of $M_{loc}^{w^*}(\R,H)$. Thus, the hull constructed will be a compact set. Moreover, the space $L^1_b(\R,H)$ is then naturally treated as a closed subspace of $M_b(\R,H)$ which consists of measures which are absolutely continuous with respect to the Lebesgue measure. For such measures $\mu_g(dt)=g(t)\,dt$ we have
$$
\|\mu_g\|_{M_b(\R,H)}=\|g\|_{L^1_b(\R,H)},
$$
so the constructed embedding is isometric. We also mention a useful fact that the space $L^1_b(\R,H)$ is {\it dense} in $M_b(\R,H)$ endowed with the local weak-star topology and that the definition of the hull of $g\in L^p_b(\R,H)$ given above is consistent with the definition of the hull $\Cal H(\mu)$ of the corresponding measure $\mu(dt)=g(t)\,dt$, see \cite{SZUMN} for more details.
\par
The class $M_b(\R,H)$ of translation-compact measures is still sometimes not convenient since it contains Dirac measures which lead to discontinuities of the corresponding trajectories of the dynamical system considered in time, so it is convenient to exclude such measures and this leads to the following definition.

\begin{definition} A measure $\mu\in M_b(\R,H)$ is called (weakly) uniformly non-atomic if, for every $\psi\in H^*$, there exists a monotone increasing function $\omega_\psi:\R_+\to\R_+$ such that $\lim_{\eb\to0}\omega(\eb)=0$
and
\begin{equation}\label{0.wna}
\sup_{t\in\R}|\int_{[t,t+\eb]}(\mu(ds),\psi)|\le\omega(\eb).
\end{equation}
The closed subset of $M_b(\R,H)$ which consists of weakly uniformly non-atomic measures is denoted by $M_b^{una}(\R,H)$. Obviously, $\Cal H(\mu)\subset M_b^{una}(\R,H)$ if $\mu\in M_b^{una}(\R,H)$ and all measures $\nu\in \Cal H(\mu)$ satisfy \eqref{0.wna} with the same $\omega_\psi$ as the initial measure $\mu$.
\end{definition}
\begin{remark} We emphasize that even in the case $\mu\in M_b^{una}(\R,H)\cap L^1_b(\R,H)$, we cannot guarantee that all measures $\nu\in \Cal H(\mu)$ will be absolutely continuous with respect to the Lebesgue measure ($\Cal H(\mu)$ is not a subspace of $L^1_b(\R,H)$ in general). Indeed, according to the Dunford-Pettis theorem, we need to assume that \eqref{0.wna} holds not only for intervals $[t,t+\eb]$, but for all sets of measure $\le\eb$. In other words, the uniform non-atomicity only guarantees that the discrete component of $\nu\in\Cal H(\mu)$ vanishes, but the singular component (Cantor-type measure) may be not zero.
\end{remark}
We now turn to the classes of external forces which may guarantee the existence of a uniform attractor in a strong topology of the phase space. The most natural and most studied  is the class of {\it translation-compact} external forces introduced by Vishik and Chepyzhov, see \cite{CV}. We recall that $g\in L^p_b(\R,H)$ is translation-compact if its hull $\Cal H(g)$ is a compact set in the {\it strong} local topology $L^p_{loc}(\R,H)$ (the definition for $\mu\in M_b(\R,H)$ is analogous). On the one hand, the translation-compactness is usually not difficult to verify using the variations of the Arzel\`{a}-Ascoli theorem and, on the other hand, this assumption is sufficient to get the attraction in strong topology for very wide class of dissipative PDEs, see \cite{CV} for more details.
\par
However, the translation compactness is far from being necessary for this and can be essentially relaxed depending on the class of PDEs considered. We start with the so-called time and space regular external forces (introduced in \cite{ZelDCDS}). Since they are somehow natural for the class of hyperbolic equations, we will give the corresponding definitions using the space of measures $M_b(\R,H)$.

\begin{definition} A measure $\mu\in M_b(\R,H)$ is {\it time-regular} if it can be approximated by smooth in time functions in the metric of $M_b(\R,H)$:
$$
\mu\in [W^{s,2}_b(\R,H)]_{M_b(\R,H)}:=M_b^{t-reg}(\R,H),\ \ s>0.
$$
As shown in \cite{ZelDCDS}, this definition is independent of $s>0$. Moreover,  any time-regular measure is uniformly non-atomic, so $M_b^{t-reg}\subset M_b^{una}$.
\end{definition}
As shown in \cite{SZUMN}, $\mu\in M_b^{t-reg}$ if and only if there exists a modulus of continuity $\omega:\R_+\to\R_+$, $\lim_{\eb\to0}\omega(\eb)=0$ such that
\begin{equation}\label{0.t-crit}
\|\mu-T(h)\mu\|_{M_b(\R,H)}\le \omega(|h|).
\end{equation}
In addition, if $\nu\in\Cal H(\mu)$ then \eqref{0.t-crit} is satisfied for $\nu$ with the same modulus of continuity. The next simple result also proved in \cite{ZelDCDS} is, however, very useful.
\begin{proposition}\label{Prop0.t-reg} Let $\mu\in M_b^{t-reg}(\R,H)$ and let $\mu_n\in W^{s,2}_b(\R,H)$ be the approximating sequence, i.e. $\|\mu-\mu_n\|_{M_b}\le \eb_n\to0$. Then, for every $\nu\in\Cal H(\mu)$, there exists $\nu_n\in\Cal H(\mu_n)\subset W^{s,2}_b(\R,H)$ such that
$$
\|\nu-\nu_n\|_{M_b}\le\eb_n.
$$
\end{proposition}

This fact, together with the standard estimate
$$
\|\Cal H(\mu_n)\|_{W^{s,2}_b}\le\|\mu_n\|_{W^{s,2}_b}
$$
show that the hull $\Cal H(\mu)$ can be {\it uniformly} approximated by smooth in time functions.
\par
We now turn to space-regular measures.
\begin{definition}\label{Def0.sreg} A measure $\mu\in M_b(\R,H)$ is space-regular if for every $\eb>0$ there exists a finite dimensional subspace $H_\eb\subset H$ and a measure $\mu_\eb\in M_b(\R,H_\eb)$ such that
\begin{equation}
\|\mu-\mu_\eb\|_{M_b(\R,H)}\le \eb.
\end{equation}
In other words the space $M_b^{s-reg}(\R,H)$ of space-regular measures is a closure in $M_b(\R,H)$ of measures with finite-dimensional ranges.
\end{definition}
As shown in \cite{ZelDCDS,SZUMN}, any measure $\mu\in M_b(\R,W)$ where $W$ is a Banach space compactly embedded in $H$ is space-regular
$$
M_b(\R,W)\subset M_b^{s-reg}(\R,H).
$$
Moreover, the absolutely continuous measure $\mu\in M_b(\R,H)$ is translation-compact if and only if it is both space and time regular:
\begin{equation}\label{0.tr-c}
L^{1,tr-c}_b(\R,H)=M_b^{t-reg}(\R,H)\cap M_b^{s-reg}(\R,H),
\end{equation}
see \cite{SZUMN} for details. However, \eqref{0.tr-c} fails in $M_b(\R,H)$. Indeed, if you take any time-periodic measure with non-zero discrete component, it will be translation-compact, but not time-regular. The reason for that failure is that, in contrast to the space $L^p(t,t+1,H)$ a general measure $\mu\in M(t,t+1,H)$ does not possess a modulus of continuity in $M(t,t+1,H)$.
\par
We also note that Definition \ref{Def0.sreg} can be essentially simplified in the case where the space $H$ possesses a Schauder base $\{e_n\}_{n=1}^\infty$. Indeed, let $P_N:H\to H_N$ be a projector to the finite dimensional space $H_N:=\operatorname{span}\{e_1,\cdots,e_N\}$:
$$
P_Nx:=\sum_{n=1}^Nx_ne_n
$$
and let $Q_N:=1-P_N$. The most important for us is the case where $H$ is a separable Hilbert space. Then we may take any orthonormal base as a Schauder base in $H$.
\begin{proposition} Let the space $H$ possess a Schauder base $\{e_n\}$. Then, the measure $\mu\in M_b(\R,H)$ is space-regular if and only if, for every $\eb>0$, there exists $N=N(\eb)$, such that
\begin{equation}\label{0.q}
\|\mu-P_{N(\eb)}\mu\|_{M_b(\R,H)}\le\eb.
\end{equation}
\end{proposition}
\begin{proof}
Indeed, in one side the statement is obvious since we may take  $H_\eb:=P_{N(\eb)}H$ and $\mu_\eb=P_{N(\eb)}\mu$. Let us verify it in the other side. Let $\mu\in M_b^{s-reg}(\R,H)$, then for every $\eb>0$, there exists a finite dimensional subspace $H_\eb\subset H$ and a measure $\mu_\eb\in M_b(\R,H_\eb)$ such that
$$
\|\mu-\mu_\eb\|_{M_b(\R,H)}\le \eb/(2L),
$$
where $L=\sup_{N}\|Q_N\|_{H\to H}$. Then, obviously,
$$
\|Q_N(\mu-\mu_\eb)\|_{M_b(\R,H)}\le \eb/2.
$$
Since any finite-dimensional subspace of a B-space is complementable, there exists a base $\{e_i\}_{i=1}^M$ in $H_\eb\subset H$ and a dual base $\{e_i^*\}_{i=1}^M\subset H^*$ such that $\<e_i^*,e_j\>=\delta_{i,j}$. Therefore,
$$
\mu_\eb=\sum_{i=1}^M\<e_i^*,\mu_\eb\>e_i.
$$
Thus, since $\lim_{N\to\infty}\|Q_Nv\|_H=0$ for all $v\in H$ and $\|\mu_\eb\|_{M_b(\R,H)}\le 2\|\mu\|_{M_b(\R,H)}$ (if $\eb>0$ is small enough), we  may fix $N=N(\eb)$ in such a way that
$$
\|Q_{N(\eb)}\mu_\eb\|_{M_b(\R,H)}\le \eb/2.
$$
Thus, $\|Q_{N(\eb)}\mu\|_{M_b(\R,H)}\le\eb/2+\eb/2=\eb$ and the proposition is proved.
\end{proof}
 \par
 The next corollary is especially useful for our purposes.
 \begin{corollary} Let $H=L^2(\Omega)$ where $\Omega$ is a bounded domain of $\R^n$ with a smooth boundary. Then,
 \begin{equation}\label{0.eggog}
 M_b^{s-reg}(\R,H)=[M_b(\R,H^1_0(\Omega))]_{M_b(\R,H)}
 \end{equation}
 and
 \begin{equation}\label{0.eggog1}
 M_b^{s-reg}(\R,H)\cap M^{una}_b(\R, H)=[M^{una}_b(\R,H^1_0(\Omega))]_{M_b(\R,H)}.
 \end{equation}
 \end{corollary}
Indeed, we may take e.g. the eigenvectors of Dirichlet-Laplacian as a Schauder base in $H$ and use $P_N\mu$ as approximations of $\mu$. Note also that \eqref{0.eggog} is proved in \cite{ZelDCDS} in a bit more general setting where $H$ is a Sobolev space which is not Hilbert. Moreover, as in Proposition \ref{Prop0.t-reg}, the approximation of $\mu\in M_b^{s-reg}$ by smooth in space functions is uniform with respect to $\nu\in\Cal H(\mu)$.
\par
We now define some kind of a mixture of space and time regular measures.
\begin{definition}\label{Def0.wreg}A measure $\mu\in M_b(\R,H)$ is weakly regular if, for every $\eb>0$,  there exist a modulus of continuity $\omega(z)$ with $\lim_{z\to0}\omega_\eb(z)=0$, a finite-dimensional subspace $H_\eb\subset H$ and a measure $\mu_\eb\in M_b(\R,H_\eb)$ such that
\begin{equation}
  \|(\mu-\mu_\eb)-T(h)(\mu-\mu_\eb)\|_{M_b(\R,H)}\le\eb+\omega_{\eb}(|h|),\ h\in\R.
\end{equation}
\end{definition}
The next proposition gives a convenient alternative way to present a weakly regular measure.
\begin{proposition}\label{Prop0.str}   For every $\eb>0$, a  weakly regular measure can be presented as a sum of a space-regular, time-regular and $\eb$-small measures. Thus,
\begin{equation}\label{0.main}
M_b^{w-reg}(\R,H)=[M_b^{t-reg}(\R,H)+M_b^{s-reg}(\R,H)]_{M_b(\R,H)}.
\end{equation}
\end{proposition}
\begin{proof} Indeed, according to the definition, for every $\eb>0$, the measure $\mu$ can be presented in the form
$$
\mu=\mu_\eb+(\mu-\mu_\eb),
$$
where $\hat\mu_\eb:=\mu-\mu_{\eb}$ is $\eb$-close to  time-regular measures and the remainder is $\eb$-small. Let us apply  the mollifying operator to the measure $\hat\mu_\eb$:
$$
({\Cal S}_h\hat\mu_\eb)(t):=\int_\R\varphi_h(s)T(t)\hat\mu_\eb(ds),
$$
where $\varphi_h$ is a symmetric normalized mollifying kernel with $\operatorname{supp}\varphi_h\subset[-h,h]$. To be more precise, the action of the
regular measure ${\Cal S}_h\hat\mu_\eb$ on a continuous function $\psi\in C(\R)$ with the finite support is defined by
\begin{multline*}
\int_\R\psi(t)\Cal S_h\hat\mu_\eb(dt):=\int_\R \psi(t)\(\int_\R \varphi_h(y-t)\hat\mu_\eb(dy)\)\,dt=\\=\int_\R\varphi_h(s)\(\int_\R \psi(y-s)\hat\mu_\eb(dy)\)\,ds=\int_\R\varphi_h(s)\(\int_\R \psi(y)T(s)\hat\mu_\eb(dy)\)\,ds,
\end{multline*}
where we have used the Fubini theorem and change of a doom variable $t\to s:=y-t$. Therefore,
$$
\int_{\R}\psi(t)\Cal S_h\hat\mu_\eb(dt)-\int_\R \psi(t)\hat\mu_\eb(dt)=\int_\R\varphi_h(s)\(\int_\R \psi(y)(T(s)\hat\mu_\eb-\hat\mu_\eb)(dy)\)\,ds
$$
and
$$
\|\hat \mu_\eb-\Cal S_h\hat\mu_{\eb}\|_{M_b(\R,H)}\le \sup_{|s|\le h}\|\hat\mu_\eb-T(s)\hat\mu_\eb\|_{M_b(\R,H)}\le \eb+\omega_\eb(|h|).
$$

Let us fix $h>0$ in such a way $\omega_{\eb}(h)\le \eb$.   Then, the measure $\mu$ can be formally written in the form:
\begin{equation}
\mu=\mu_\eb+\Cal S_h\hat\mu_\eb+(\hat\mu_\eb-\Cal S_h\hat\mu_\eb).
\end{equation}
It only remains to note that the measure $\mu_\eb$ is space-regular, the measure $\Cal S_h\hat \mu_\eb$ is smooth in time and therefore is time-regular and the remainder  satisfies
$$
\|\hat\mu_\eb-\Cal S_h\hat\mu_\eb\|_{M_b(\R,H)}\le 2\eb.
$$
Thus, the proposition is proved.
\end{proof}
\begin{remark} The typical for hyperbolic equations assumption is that the
 measure $\mu$ is weakly regular and weakly non-atomic simultaneously. We do not know whether or not such a measure can be approximated by the sums of time-regular and space-regular measures both of them is weakly non-atomic:
 \begin{equation}\label{0.good0}
 M^{w-reg}_b(\R,H)\cap M^{una}_b(\R,H)=[M^{t-reg}_b(\R,H)+(M^{s-reg}_b(\R,H)\cap M^{una}_b(\R,H))]_{M_b(\R,H)}.
 \end{equation}
 However, the answer on this question is clearly positive in the case where $H$ possesses a Schauder base. Indeed, in this case we may take $H_\eb=P_{N_\eb}H$ and $\mu_\eb=P_{N_\eb}\mu$. Thus, the last measure is clearly non-atomic if the initial $\mu$ is non-atomic.
 \par
 Combining \eqref{0.good0} with \eqref{0.eggog1} and Proposition \ref{Prop0.t-reg}, we end up with the following useful characterization of weakly regular non-atomic measures for the case $H=L^2(\Omega)$:
 \begin{equation}\label{0.good}
 M^{w-reg}_b(\R,H)\cap M^{una}_b(\R,H)=[W^{1,2}_b(\R,H)+ M^{una}_b(\R,H^1_0(\Omega))]_{M_b(\R,H)}.
 \end{equation}
Moreover, as in the case of space-regular or time-regular measures, these approximations are uniform with respect to all $\nu\in\Cal H(\mu)$.
\end{remark}
To complete this section, we briefly discuss the classes of external forces adapted to parabolic PDEs which have been introduced in \cite{13,LWZ,14,15,ZelDCDS}. These classes will be used in the sequel in our study of Navier-Stokes equations.
\begin{definition}\label{Def0.norm} A function $g\in L^p(\R,H)$, $1<p<\infty$, is normal if there exists a modulus of continuity $\omega:\R_+\to\R_+$ such that $\lim_{z\to0}\omega(z)=0$ and
\begin{equation}\label{0.norm}
\sup_{t\in\R}\int_t^{t+h}\|g(s)\|^p_H\,ds\le \omega(|h|),\ \ h\in\R.
\end{equation}
A function $g\in L^p(\R,H)$ is weakly normal if, for any $\eb>0$, there exists a finite dimensional subspace $H_\eb\subset H$, a modulus of continuity $\omega_\eb(z)$ and a function $g_\eb\in L^p_b(\R,H_\eb)$ such that
\begin{equation}\label{0.w-norm}
\sup_{t\in\R}\int_t^{t+h}\|g(s)-g_\eb(s)\|^p_H\,ds\le \eb+\omega_\eb(|h|).
\end{equation}
\end{definition}
The key properties of normal and weakly normal functions are collected in the following proposition, see \cite{ZelDCDS,LWZ} for the proof.
\begin{proposition} Let $g\in L^p_b(\R,H)$ be normal (i.e. $g\in L^{p,norm}_b(\R,H)$). Then
\begin{equation}\label{0.exp-norm}
\lim_{L\to\infty}\sup_{t\in\R}\int_{-\infty}^te^{-L(t-s)}\|g(s)\|^p_H\,ds=0.
\end{equation}
Let $g\in L^p_b(\R,H)$ be weakly normal (i.e. $g\in L^{p,w-norm}_b(\R,H)$). Then, for every $\eb>0$, there exists a finite dimensional subspace $H_\eb\subset H$ and a function $g_\eb\in L^p_b(\R,H_\eb)$  such that
\begin{equation}\label{0.exp-wnorm}
\limsup_{L\to\infty}\sup_{t\in\R}\int_{-\infty}^te^{-L(t-s)}\|g(s)-g_\eb(s)\|^p_H\,ds\le\eb.
\end{equation}
Moreover, the limits \eqref{0.exp-norm} and \eqref{0.exp-wnorm} are uniform with respect to all $h\in\Cal H(g)$.
\end{proposition}
\begin{remark}\label{Rem0.w-norm} In the case where $H$ possesses a Schauder base we may take $H_\eb:=P_{N(\eb)}H$ and $g_\eb(t)=P_{N(\eb)}g(t)$ in the Definition \ref{Def0.norm}, In this case, property \eqref{0.exp-wnorm} can be reformulated in a more convenient form: if $g\in L^{p,w-norm}_b(\R,H)$, then, for every $\eb>0$ there exist $N(\eb)$ and $L(\eb)$ such that
\begin{equation}\label{0.exp-w-norm1}
\sup_{t\in\R}\int_{-\infty}^te^{-L(t-s)}\|Q_{N}g(s)\|^p_H\,ds\le \eb
\end{equation}
for all $N\ge N(\eb)$ and $L\ge L(\eb)$.  Moreover, this estimate is also uniform with respect to all $h\in\Cal H(g)$.
\end{remark}

\section{Uniform attractors and entropy: an abstract scheme}\label{s3.m}
In this section, we state and prove the main theorem which will allow us in a sequel to get entropy estimates for various classes of dissipative PDEs. We first recall some known constructions related with uniform attractors.
\par
Let $E$ and $H$ be two reflexive B-spaces and let $M_b^{una}(\R,H)$ be the space of uniformly (weakly) non-atomic $H$-valued measures endowed by the weak-star topology, see section \S\ref{s.2n}  for the definitions. Then, the group of shifts
\begin{equation}\label{1.Th}
T(h): M_b^{una}(\R,H)\to M_b^{una}(\R,H),\ \ \ T(h)\mu (t):=\mu(t+h),\ \ t,h\in\R
\end{equation}
acts continuously. Moreover, the restriction operators $\mu\to\mu|_{t\in[a,b]}$ are also continuous as maps from $M_b^{una}(\R,H)$ to $M^{una}([a,b],H)$ in the weak star topology. For any element $\mu\in M_b^{una}(\R,H)$, we define its hull $\Cal H(\mu)$ via Definition \ref{Def0.hull}, where the closure is taken in the topology of $M_{loc}^{w^*}(\R,H)$.
Then $\Cal H(\mu)$ is a compact set in  $M_b^{una}(\R,H)$, see \cite{SZUMN}.
\par
Let us further assume that, for every $\nu\in\Cal H(\mu)$, we are given a dynamical process
 $U_\nu(t,\tau): E\to E$, $t\ge\tau\in\R$. The latter means that these operators satisfy the Markovian property
 \begin{equation}
 U_\nu(\tau,\tau)=\operatorname{Id},\ \ U_\nu(t,\tau)=U_\nu(t,s)\circ U_\nu(s,\tau),\ \ t\ge s\ge\tau\in\R.
 \end{equation}
 Assume also that the family of processes $\{U_\nu(t,\tau),\nu\in\Cal H(\mu)\}$ satisfies the translation identity
 \begin{equation}\label{1.trans}
 U_{\nu}(t+h,\tau+h)=U_{T(h)\nu}(t,\tau),\ \ t\ge\tau,\ \ \tau\in\R,\ \ h\in\R.
 \end{equation}
 Then we may define an extended semigroup (cocycle)  $\mathbb S(t)$ acting on $\Bbb E:=E\times\Cal H(\mu)$ via
 \begin{equation}
 \Bbb S(t)(\xi,\nu):=(U_\nu(t,0)\xi,T(t)\nu),\ \ t\ge0,\ \ \xi\in E,\ \ \nu\in\Cal  H(\mu)
 \end{equation}
 and study the global attractor of this semigroup in the appropriate topology.
\begin{definition}\label{Def1.attr} A set $\Bbb A^w$ is a weak global attractor of the semigroup $\Bbb S(t)$ if
\par
1) $\Bbb A^w$ is compact in $\Bbb E^w:=E^w\times\Cal H(\mu)$, where $E^w$ stands for the space $E$ endowed with the weak topology;
\par
2) $\Bbb A^w$ is invariant: $\Bbb S(t)\Bbb A^w=\Bbb A^w$, $t\ge0$;
\par
3) $\Bbb A^w$ attracts the images of bounded sets in $\Bbb E$ in the topology of $\Bbb E^w$, i.e., for any $\Bbb B$ bounded in $\Bbb E$ and any neighbourhood $\Cal O(\Bbb A^w)$ in the topology of $\Bbb E^w$ there exists $T=T(\Cal O,\Bbb B)$ such that
$$
\Bbb S(t)\Bbb B\subset\Cal O(\Bbb A^w),\ \ t\ge T.
$$
The strong attractor $\Bbb A^s$ is defined analogously, but the  weak topology on $E$ should be replaced by the norm topology. The weak star topology in $\Cal H(\mu)$ remains unchanged. Note that in both cases the amount of bounded sets remains the same: the set $\Bbb B$ is bounded in $\Bbb E^w$ or in
$\Bbb E$ if and only if its projection $\Pi_1\Bbb B$ to the first component is bounded in $E$.
\par
Finally the (weak, strong) uniform attractor for the family of processes $\{U_\nu(t,\tau),\nu\in\Cal H(\mu)\}$ is nothing more than the projection of the corresponding global attractor to the first component:
$$
\Cal A_{un}^s=\Pi_1\Bbb A^s,\ \  \Cal A_{un}^w=\Pi_1\Bbb A^w.
$$
\end{definition}
\begin{remark} It is well-known that $\Bbb A^s=\Bbb A^w$ if both attractors exist (and that the existence of a strong attractor implies the existence of a weak one). Moreover, the uniform attractor $\Cal A^s_{un}$ possesses an intrinsic definition without referring to the extended semigroup. Namely, the set $\Cal A^s_{un}$ is a strong uniform attractor for the family $\{U_\nu(t,\tau),\nu\in\Cal H(\mu)\}$ if
\par
1) The set $\Cal A_{un}^s$ is compact in $E$;
\par
2) It attracts uniformly with respect to $\nu\in \Cal H(\mu)$ the images of all bounded sets, i.e., for every bounded set $B\subset E$ and every neighbourhood $\Cal O(\Cal A^s_{un})$ of $\Cal A^s_{un}$ in $E$ there exists $T=T(\Cal O,B)$ such that
$$
\cup_{\nu\in\Cal H(\mu)}U_\nu(t,0)B\subset \Cal O(\Cal A^s_{un}),\ \forall t\ge T;
$$
\par
3) $\Cal A^s_{un}$ is a minimal set which satisfies the above two properties.
\par
The above alternative definition holds for the weak attractor $\Cal A^w_{un}$ as well (with the natural replacement of the norm topology by the weak one), see \cite{CV,ZelDCDS} for more details.
\par
There is also a natural choice of strong topologies in both components $E$ and $\Cal H(\mu)$, but in this case we need the compactness of the hull $\Cal H(\mu)$ in the strong topology as well which requires the translation-compactness of the initial external force $\mu$. Since we are mainly interested in the case when this condition is violated, we do not consider this choice here.
\end{remark}
The next standard result gives the conditions for the existence of the attractors introduced above.
\begin{theorem}\label{Th1.exist} Let the family of processes $\{U_\nu(t,\tau),\nu\in\Cal H(\mu)\}$ satisfy the above assumptions. Assume also that this family is uniformly dissipative in $E$, i.e. that
\begin{equation}\label{1.dis}
\|U_\nu(t,\tau)\xi\|_E\le Q(\|\xi\|_E)e^{-\alpha(t-\tau)}+C_*,\ \ t\ge\tau,\  \nu\in\Cal H(\mu),\ \ \xi\in E
\end{equation}
holds uniformly with respect to $\nu$, $t$ and $\tau$. Here $\alpha>0$ and $Q$ is monotone increasing function both independent of $\xi$, $\nu$, $t$ and $\tau$.
\par
Then the family $\{U_\nu(t,\tau),\nu\in\Cal H(\mu)\}$ possesses a weak uniform attractor $\Cal A^w_{un}\subset {E}$.
\par
If, in addition, this family is uniformly asymptotically compact on bounded sets, i.e., for every $B$ bounded in $E$ and  every sequences $\xi_n\in B$, $\nu_n\in\Cal H(\mu)$, $t_n\to\infty$, the sequence
$$
\{U_{\nu_n}(t_n,0)\xi_n,\ n\in\Bbb N\}
$$
is precompact in $E$, then the weak uniform attractor $\Cal A^w_{un}$ coincides with the strong one $\Cal A^s_{un}=\Cal A^w_{un}$.
\end{theorem}
For the proof of this theorem, see \cite{CV}. We mention here that the weak part of this theorem is strongly based on the dissipativity estimate and the Banach-Alaoglu theorem which gives us the existence of a compact uniformly absorbing set in $\Bbb E^w$, mention also that the asymptotic compactness is actually necessary to verify on this absorbing set only.
\par
The next standard result gives us the structure of a uniform attractor in the case where the family of processes is continuous.
\begin{corollary}\label{Cor1.rep} Let the family of dynamical processes $\{U_\nu(t,\tau),\nu\in\Cal H(\mu)\}$ be uniformly dissipative (i.e. \eqref{1.dis} holds) and the maps $(\xi,\nu)\to U_\nu(t,\tau)\xi$ be continuous as the maps from $\Bbb E^w$ to $E^w$ for every fixed $t$ and $\tau$. Then the attractor $\Cal A_{un}^w$ possesses the following description:
\begin{equation}\label{1.rep}
\Cal A_{un}^w:=\cup_{\nu\in\Cal H(\mu)}\Cal K_\nu,
\end{equation}
where
$$
\Cal K_\nu:=\{u:\R\to E,\ \ U_\nu(t,\tau)u(\tau)=u(t),\ t\ge\tau\in\R\}
$$
is a set of all complete bounded trajectories of the process $U_\nu(t,\tau)$ (the so-called kernel of $U_\nu(t,\tau)$ in the terminology of Vishik and Chepyzhov, see \cite{CV}).
\end{corollary}
The proof of this corollary can be found e.g. in \cite{CV}.
\begin{remark} Representation formula \eqref{1.rep} is crucial for what follows, so we will always assume in the sequel that the continuity assumptions of Corollary \ref{Cor1.rep} are satisfied. Note that this representation may be lost without the continuity assumptions (see \cite{SZUMN}), although it remains true if the continuity is replaced by the assumption that $(\xi,\nu)\to U_{\nu}(t,\tau)\xi$ have closed graphs, see \cite{Pata}.
\end{remark}
We now turn to the entropy estimates. For the convenience of the reader we first recall the definition of Kolmogorov's $\eb$-entropy, see \cite{KT,MirZel} for more detailed expositions and typical examples.
\begin{definition} Let $K$ be a compact set in a metric space $E$. Then, by the Hausdorff criterion, for any $\eb>0$, it can be covered by finitely many $\eb$-balls in $E$. Let $N_\eb(K,E)$ be the minimal balls which cover $K$. Then, the Kolmogorov's $\eb$-entropy $\Bbb H_\eb(K,E)$ of $K$ in $E$ is defined as
$$
\Bbb H_\eb(K,E):=\log_2N_\eb(K,E).
$$
\end{definition}
The next theorem gives us the main technical tool for obtaining the upper bounds for the Kolmogorov's entropy of the attractor.
\begin{theorem}\label{Th1.main-ent} Let the assumptions of Corollary \ref{Cor1.rep} hold, let $\Cal B$ be a bounded uniformly absorbing ball for the family $\{U_\nu(t,\tau)$, $\nu\in\Cal H(\mu)\}$, and let $s>0$ be such that
\begin{equation}\label{1.abs}
U_\nu(s,0)\Cal B\subset \Cal B,\ \ \forall\nu\in\Cal H(\mu).
\end{equation}
Assume also that, for every $\eb>0$, there exists a map
\begin{equation}\label{1.pr}
\Pr(\eb): \Cal H(\mu)\to M_b^{una}(\R,H),\ \  \Pr(\eb)\circ T(h)=T(h)\circ\Pr(\eb)
\end{equation}
such that the image $\Pr(\eb)\Cal H(\mu)$ is  precompact in the strong topology of $M_b^{una}(\R,H)$ and such that, for every $\xi_1,\xi_2\in\Cal B$ and every $\nu_1,\nu_2\in\Cal H(\mu)$, the following estimate holds:
\begin{multline}\label{1.squeez}
\|U_{\nu_1}(s,0)\xi_1-U_{\nu_2}(s,0)\xi_2\|_{E}\le{\frac1{16}}\|\xi_1-\xi_2\|_{E}+
L\|\xi_1-\xi_2\|_W+\\+\eb+C\|\Pr(\eb)\nu_1-\Pr(\eb)\nu_2\|_{M_b([0,s],H)},
\end{multline}
where $W$ is another B-space such that the embedding $E\subset W$ is compact and the constants $L$ and $C$ are independent of $\eb$ and $\xi_1,\xi_2,\nu_1,\nu_2$.
\par
Then the family of dynamical processes $\{U_\nu(t,\tau),\nu\in\Cal H(\mu)\}$ possesses a strong uniform attractor $\Cal A_{un}^s$ in $E$ and the Kolmogorov's $\eb$-entropy of this attractor possesses the following upper bounds:
\begin{multline}\label{1.ent-est}
\Bbb H_\eb(\Cal A^s_{un},E)\le C_1\log_2(\frac{2R_0}\eb)+\\+\Bbb H_{\eb/K_1}\(\Pr(\eb/K_1)\Cal H(\mu),M_b([0,L_1\log_2(\frac{2R_0}\eb)],H)\),\ \eb\le R_0,
\end{multline}
for some new constants $C_1$, $K_1$, $L_1$ and $R_0$ which are independent of $\eb\to0$.
\end{theorem}
\begin{proof}Let us first verify the asymptotic compactness. Indeed, let
$$
\omega_t(\Cal B)=\bigg[\cup_{\nu\in\Cal H(\mu)}\cup_{\tau\ge t}U_\nu(\tau,0)\Cal B\bigg]_{E}.
$$
Then, since $\omega_{t_1}(\Cal B)\subset\omega_{t_2}(\Cal B)$ if $t_1\ge t_2$, estimate \eqref{1.squeez} gives us the following upper bound for the Kuratowski measure of non-compactness:
$$
\kappa(\omega_t(\Cal B))\le {\frac C{16^n}},\ \ t\ge ns,\ \ n\in\Bbb N.
$$
Thus, $\lim_{t\to\infty}\kappa(\omega_t(\Cal B))=0$ and, therefore, the family $\{U_\nu(t,\tau),\nu\in\Cal H(\mu)\}$ is indeed uniformly asymptotically compact, so the strong uniform attractor exists and can be found by
$$
\Cal A^s_{un}=\cap_{t\ge0}\omega_t(\Cal B).
$$

 Let us now verify the entropy estimate. To simplify the notation, we will write from now on $\Cal A$ and $U_\nu(t)$ instead of $\Cal A_{un}^s$ and $U_\nu(t,0)$ respectively. Then, from the translation identity, we have
 $$
 U_\nu(ns)=U_{T((n-1)s)\nu}(s)\circ \cdots \circ U_\nu(s).
 $$
 Using estimates \eqref{1.squeez}, we extend the construction of  $\varepsilon$-covering for the uniform attractor given  \cite{ZeDCDS} to the non-translation compact set.
Let $B(\eb_0,\xi_0^i,E)$, $i=1,\cdots,N_0(\eb_0)$ be the initial $\eb_0$-covering of $\Cal B$. For instance, since $\Cal A\subset\Cal B$ is bounded, we may take $\eb_0=R_0$ for sufficiently large $R_0$ and $N_0(\eb_0)=1$.

 Let us fix also $k\in\Bbb N$ and the $\eb_0/(C2^{k+3})$-covering of the set $\Pr(\eb_0/2^{k+3})\Cal H(\mu)\big|_{[0,ks]}$ in the space $M_b([0,ks],H)$.
  Let $\nu_j$, $j=1,\cdots,N(\eb_0/(C2^{k+3}),\mu)$ be the centers of this covering.

Let $\xi\in\Cal A$, then, due to the representation formula \eqref{1.rep}, there exist $\xi_0\in\Cal A$ and $\eta\in\Cal H(\mu)$ such that $\xi=\xi_k=U_\eta(ks)\xi_0$. We can find the indexes $i$ and $j$ such that
\begin{equation}\label{1.net}
\|\xi_0-\xi_0^i\|_E\le\eb_0,\ \ \ \|\Pr(\eb_0/2^{k+3})\eta-\Pr(\eb_0/2^{k+3})\nu_j\|_{M_b([0,ks],H)}\le\eb_0/(C2^{k+3}).
\end{equation}
Let us consider the balls $B(\eb_0,\xi_0^i,E)$ which cover the attractor $\Cal A$. Every of these balls
is a (pre)compact set in $W$ since the embedding $E\subset W$ is compact and, therefore, can be covered by finitely many $\eb_0/{(16L)}$-balls in $W$. Let $v_0^{i}$ be the centers of these new balls. Moreover, increasing the radii of these balls twice if necessary and drop out the balls which do not intersect with the absorbing set $\Cal B$, we may assume without loss of generality that $v_{0}^{i}\in\Cal B$. Thus, we have constructed an $\eb_0/(8L)$-covering  in the metric of $W$. We also note that the number of balls in $W$ which is necessary to cover every of the balls $B(\eb_0,\xi^i_0,E)$ can be estimated by the following number:
\begin{equation}
N_{\eb_0/(16L)}(B(\eb_0,\xi_0^i,E),W)=N_{\eb_0/(16L)}(B(\eb_0,0,E),W)=N_{1/(16L)}(B(1,0,E),W):=N.
\end{equation}
Crucial for us that the number $N$ is independent of $\eb_0$ and $i$. Thus, the number of balls in the constructed $\eb_0/(8L)$-covering of $\Cal B$ in $W$ does not exceed $NN_0(\eb_0)$.
We also note that, by the construction, $\{\xi_0^{i}\}$ is  an $\eb_0$-net of $\Cal B$ in $E$ and $\{v_0^{i}\}$ is still a $2\eb_0$-net of $\Cal B$ in $E$.

Let $\xi_l=U_\eta(ls)\xi_0$, $l=1,\cdots,k$ and $\xi_1^{i,j}:=U_{\nu_j}(s)v_0^{i}$.
 Then according to the estimate \eqref{1.squeez} and inequalities  \eqref{1.net}, there are exist the indexes ${i,j}$ such that
\begin{multline}
\|\xi_1-\xi_1^{i,j}\|_E\le\frac1{16}\|\xi_0-v_0^{i}\|_E+\eb_0/2^{k+3}+L\|\xi_0-v_0^{i}\|_W+\\+
C\|\Pr(\eb_0/2^{k+3})\eta-\Pr(\eb_0/2^{k+3})\nu_j\|_{M_b([0,s],H)}\le \eb_0/8+\eb_0/8+\eb_0/8+\eb_0/8\le \eb_0/2.
\end{multline}
Note that the number of points in the constructed 1st order $\eb_0/2$-net $\{\xi_1^{i,j}\}$ does
 not exceed
 $$
 N_1(\eb_0/2)\le NN_0(\eb_0)N(\eb_0/(C2^{k+3}),\mu).
 $$

 We now repeat the above described procedure starting from the constructed $\{\xi_1^{i,j}\}$ to get the 2nd order $\eb_0/4$-net in the metric of $E$. Namely, as before, we first construct the $\eb_0/{(16L)}$-covering of every $\eb_0/2$-ball centered in $\xi_1^{i,j}$ with the centers $v_1^{i,j}$ belonging to the absorbing set $\Cal B$ in metric $W$. As before, this will be also an {$\eb_0/2$}-net of $\cup_{\nu\in\Cal H(\mu)}U_\nu(s)\Cal B$ in the metric of $E$. Here and below we keep index $j$ indicating the covering of the ball $B(\eb_0/2, \xi_1^{i,j},E)$. 
 \par
 After that we set $\xi_2^{i,j}:=U_{T(s)\nu_j}(s)=U_{\nu_j}(2s,s)v_1^{i,j}$ and again, according to estimate \eqref{1.squeez}, we can find indexes $\{i,j\}$ such that
 \begin{multline}
 \|\xi_2-\xi_2^{i,j}\|_E\le\frac1{16}\|\xi_1-v_1^{i,j}\|_E+\eb_0/2^{k+3}+L\|\xi_1-v_1^{i,j}\|_W+\\+
C\|\Pr(\eb_0/2^{k+3})\eta-\Pr(\eb_0/2^{k+3})\nu_j\|_{M_b([s,2s],H)}\le \\\le \eb_0/16+\eb_0/16+\eb_0/16+\eb_0/16\le \eb_0/4.
 \end{multline}
 Thus, the second order $\eb_0/4$-net (which actually gives the covering of $\cup_{\nu\in\Cal H(\mu)}U_\nu(2s)\Cal B$) is constructed and the number of points in this covering does not exceed
 $$
 N_2(\eb_0/4)\le N^2N_0(\eb_0)N(\eb_0/(C2^{k+3}),\mu).
 $$

 Repeating this procedure, we finally arrive at the $k$th order net $\{\xi_{k}^{i,j}\}$ which gives the $\eb_0/2^k$-covering of the set $\cup_{\nu\in\Cal H(\mu)}U_\nu(ks)\Cal B$. In particular, by the representation formula any point $\xi=\xi_k$ belongs to this set and, therefore, this is the desired $\eb_0/2^k$-net of the attractor $\Cal A$ in $E$. The number of balls in this covering is estimated by
 $$
 N_k(\eb_0/2^k)\le N^kN_0(\eb_0)N(\eb_0/(C2^{k+3}),\mu).
 $$

 We are now ready to finish the proof of the theorem. Indeed, taking the logarithm from both parts of the last inequality and using that $N_0(\eb_0)=1$, $\eb_0=R_0$, we arrive at
 \begin{multline}\label{1.end}
 \Bbb H_{R_0/2^k}(\Cal A, E)\le k\log_2N+\log_2N(R_0/(C2^{k+3}),\mu)=\\=k\log_2N+\Bbb H_{R_0/(C2^{k+3})}(\Pr(R_0/2^{k+3})\Cal H(\mu),M_b([0,ks],H)).
\end{multline}
  We claim that \eqref{1.end} implies the desired estimate \eqref{1.ent-est}. Indeed,
  consider an arbitrary $\eb\le\eb_0$ and choose $k\in\N$ so that
\begin{equation*}
\frac{R_0}{2^{k-1}}\geq\eb\geq\frac{R_0}{2^k}.
\end{equation*}
Then, $\log_2\big(\frac{R_0}{\eb}\big)\le k\le\log_2\big(\frac{2R_0}{\eb}\big)$ and putting $2^k\sim \frac{2R_0}{\eb}$ into the right-hand side of \eqref{1.end}, we arrive at
\begin{multline}
\Bbb H_\eb(\Cal A,E)\le\Bbb H_{\eb_0/2^k}(\Cal A,E)\le
\log_2N\log_2\big(\frac{2R_0}{\eb}\big)
+\\+\mathbb H_{\eb/(16C)}\(\Pr(\eb/16)\Cal H(\mu),M_b([0,s\log_2(\frac{2R_0}{\eb})],H)\).
\end{multline}
Note that replacing $2^k$ by $\frac{2R_0}{\eb}$ in the last term of \eqref{1.end} may look a bit ambiguous since we a priori do not have any monotonicity for this term. However, this is not a problem since analyzing the proof given above, we see that $2^k$ in this expression can be replaced by any number from the interval $[2^{k-1},2^k]$ (and, in particular, by {$2R_0/\eb$}) without destroying the estimates. Thus, Theorem \ref{Th1.main-ent} is proved.
\end{proof}
\begin{remark}\label{Rem1.NS} Note that the factor {$\frac1{16}$} is taken in \eqref{1.squeez} for simplicity only. A bit more accurate arguments show that any number $\kappa<1$ is admissible there. This is not important for our results, but may be crucial for the accurate analysis of the dependence of the fractal dimension of the attractors via estimate \eqref{1.ent-est}. Indeed, as known,  best estimate for it are often requires time $s$ to be small and in this case the contraction factor $\kappa$ should be close to one.
\par
More important for us is the case where $\Cal H(\mu)$ is more regular, say $\Cal H(\mu)\subset L^p_b(\R,H)$ and we have the analogue of estimate \eqref{1.squeez} with the $L_b^p([0,s],H)$-space only. In this case repeating word by word the proof of the theorem given above, we get the analogue of estimate \eqref{1.ent-est} with the space $M_b([0,s\log_2(\frac{2R_0}\eb)],H)$ replaced by $L_b^p([0,s\log_2(\frac{2R_0}\eb)],H)$.
\par
Note also that in the proved theorem we do not specify the nature of the projection operators $\Pr(\eb)$, in particular, they may be non-linear or/and discontinuous. However, in applications, they usually will be some kind of space, time or space-time smoothing linear operators.
\end{remark}

\section{Application 1. Subcritical damped wave equation with Dirichlet BC}\label{s4}

In this section, we apply the abstract result proved above to the following weakly damped wave equation:
\begin{equation}\label{2.wave}
\Dt^2 u+\gamma\Dt u-\Dx u+u+f(u)=\mu(t),\ \ \xi_u\big|_{t=\tau}=\xi_\tau,\ \ u\big|_{\partial\Omega}=0
\end{equation}
in a bounded smooth domain $\Omega\subset\R^3$ endowed with the Dirichlet boundary conditions. Here and below $\xi_u(t)$ stands for the pair of functions $\{u(t),\Dt u(t)\}$, $\gamma>0$, $f(u)\in C^1(\R)$ is a given non-linearity which satisfies the following dissipativity and growth conditions:
\begin{equation}\label{2.f}
1.\   f(u)u\ge -C,\ \ 2.\  F(u)\le \alpha f(u)u+C,\ \ 3.\  |f'(u)|\le C(1+|u|^{p-1})
\end{equation}
 with $F(u):=\int_0^uf(v)\,dv$ and some positive constants $C$ and $\alpha$ and the growth exponent $p$ satisfying  $1\le p<5$. Finally, $\mu(t)$ is a given external force which can be an $H:=L^2(\Omega)$ valued measure with respect to time. Namely, following \cite{SZUMN}, we assume that
 \begin{equation}\label{2.mu}
 \mu\in M_b^{una}(\R,H).
 \end{equation}
We take the standard energy space $E:=H^1_0(\Omega)\times L^2(\Omega)$ as the phase space for problem \eqref{2.wave}. In particular, we assume that $\xi_\tau\in E$.
\par
The well-posedness and dissipativity of problem is verified in \cite{SZUMN} in the class of the so-called Shatah-Struwe (SS) solutions, see \cite{BSS,11,60,61}. We recall that $u(t)$ is a SS-solution of problem \eqref{2.wave} if $\xi_u\in C([\tau,T],E)$ and
\begin{equation}\label{2.str}
u\in L^4(\tau,T; L^{12}(\Omega)),\ \ \text{for all}\ \ T>\tau
\end{equation}
and if it satisfies equation \eqref{2.wave} in the sense of distributions.

\begin{theorem}\label{Th2.wp} Let the above assumptions hold. Then, for every $\tau\in\R$ and $\xi_\tau\in E$, problem \eqref{2.wave} possesses a unique SS-solution $u(t)$. This solution satisfies the following dissipative energy-Strichartz estimate:
\begin{equation}\label{2.en-str-dis}
\|\xi_u(t)\|_{E}+\|u\|_{L^4(\max\{\tau,t-1\},t;L^{12}(\Omega))}\le Q(\|\xi_\tau\|_E)e^{-\alpha (t-\tau)}+Q(\|\mu\|_{M_b(\R,H)})
\end{equation}
for some monotone increasing function $Q$ and positive $\alpha$ which are independent of $u$, $\mu$, $\tau\in\R$ and $t\ge\tau$. Moreover, the following energy identity holds:
\begin{multline}\label{2.en}
\frac12\frac d{dt}\(\|\Dt u(t)\|^2_H+\|\Nx u(t)\|^2_H+\|u(t)\|^2_H+2(F(u(t)),1)-2\int_\tau^t(\Dt u(s),\mu(ds))\)+\\+\|\Dt u(t)\|^2_H=0
\end{multline}
holds for (almost) all $t\ge\tau$.
\end{theorem}
The proof of this theorem is given in \cite{SZUMN} in a more complicated case of critical non-linearity $f$ and periodic boundary conditions and is strongly based on the following energy-to-Strichartz estimate for the SS-solutions of \eqref{2.wave}:
\begin{equation}\label{2.en-str}
\|u\|_{L^4(t,t+1;L^{12}(\Omega))}\le Q(\|\xi_u(t)\|_E)+Q(\|\mu\|_{M([t,t+1],H)})
\end{equation}
for some monotone function $Q$ which is independent of $t$, $\mu$ and $u$. This estimate can be deduced from the analogous estimate for the linear equation ($f=0$) by elementary perturbation  methods if $f(u)$ is subcritical (which is assumed in this section), see \cite{SZUMN} for any reasonable boundary conditions, but becomes rather delicate in the critical case. Actually, for the case where $f$ has a quintic growth rate, it is known for periodic BC or the case of whole space only, see \cite{SZUMN,MSSZ} and also \cite{BG,Tao} (or \cite{BLP,11} for the {\it autonomous} critical case)  and this is the main reason why we assume that $p<5$ in this section.
\par
Note also that the assumption for the measure $\mu$ to be non-atomic is necessary in order to have $\Dt u\in C(t,t+1;H)$. As shown in \cite{SZUMN} without this assumption the trajectories of \eqref{2.wave} may be discontinuous in time, the corresponding dynamical process also becomes discontinuous with respect to the external force $\mu$ and, as a result, the key representation formula \eqref{1.rep} may also be lost.
\par
Let us now turn to the attractors. Following the general scheme, see e.g. \cite{CV}, we consider the family of equations of the form \eqref{2.wave}
\begin{equation}\label{2.wave1}
\Dt^2 u+\gamma\Dt u-\Dx u+u+f(u)=\nu, \ \xi_u\big|_{t=\tau}=\xi_\tau,\ u\big|_{\partial\Omega}=0
\end{equation}
with external forces $\nu$ belonging to the hull $\Cal H(\mu)$ of the initial external force $\mu$:
\begin{equation}\label{2.hull}
\nu\in\Cal H(\mu):=\big[T(h)\mu,\ h\in\R\big]_{M_{loc}^{w^*}(\R,H)}.
\end{equation}
Then, according to Theorem \ref{Th2.wp}, the solution operators of equations \eqref{2.wave1} generate a family of dynamical processes $\{U_\nu(t,\tau),\nu\in\Cal H(\mu)\}$ in the energy phase space $E$ which satisfy the translation identity  \eqref{1.trans} as well as the dissipative estimate \eqref{1.dis} (thanks to \eqref{2.en-str-dis}). Moreover, the hull $\Cal H(\mu)$ is compact in the weak star topology by the Banach-Alaoglu theorem and as it is straightforward to verify the maps
$(\xi,\nu)\to U_\nu(t,\tau)\xi$ are continuous as maps from $E^w\times\Cal H(\mu)$ to $E^w$ for every fixed $t$ and $\tau$. Thus, due to Theorem \ref{Th1.exist} and Corollary \ref{Cor1.rep}, we have the following result.
\begin{corollary}\label{Cor2.weak} Let the assumptions of Theorem \ref{Th2.wp} hold. Then the family $\{U_\nu(t,\tau),\nu\in\Cal H(\mu)\}$ possesses a weak uniform attractor $\Cal A_{un}^w$ in $E$ (see Definition \ref{Def1.attr}) and the representation formula \eqref{1.rep} holds.
\end{corollary}
The main aim of this section is to verify that the constructed attractor is actually strong and to get the upper bounds for its Kolmogorov's entropy. It is known (see counterexamples in \cite{ZelDCDS}) that condition \eqref{2.mu} is not enough for the compactness of $\Cal A^w_{un}$ in a strong topology of $E$, so we need to put some extra assumptions on the initial external force $\mu$. Namely, we will assume in addition that the measure $\mu$ is weakly regular in the sense of Definition \ref{Def0.wreg}:
\begin{equation}\label{2.w-r}
\mu\in M_b^{w-reg}(\R,H)\cap M_b^{una}(\R,H).
\end{equation}
We are planning to use the abstract Theorem \ref{Th1.main-ent}, so we need to specify the projector operators $\Pr(\eb)$. To this end, we introduce the smoothing in time operators $S_\delta$ by
  \begin{equation}\label{2.S}
  (S_\delta h)(t):=\frac1{\delta}\int_t^{t+\delta}h(s)\,ds
  \end{equation}
  and will use the following projectors
  \begin{equation}\label{2.pr}
  \Pr (\eb):=S_{\delta(\eb)}\circ P_{N(\eb)}
  \end{equation}
where $\delta(\eb)$ is small enough and $N(\eb)$ is large enough.  Here and below we fix an orthonormal base $\{e_n\}_{n=1}^\infty$ generated by the eigenvectors of the Dirichlet-Laplacian in $H$ and use the orthoprojectors $P_N$ and $Q_N$ which correspond to the first $N$ eigenvectors and the rest of them respectively.
\par
We are now ready to state and prove the main result of this section.

\begin{theorem}\label{Th2.main-ent} Let the assumptions of Theorem \ref{Th2.wp} hold and let, in addition, the initial external force $\mu$ satisfy \eqref{2.w-r}. Then the family of processes $\{U_\nu(t,\tau),\nu\in\Cal H(\mu)\}$ generated by equations \eqref{2.wave1} possess a strong uniform attractor $\Cal A_{un}=\Cal A^s_{un}=\Cal A_{un}^w$. Moreover, its Kolmogorov's entropy possesses the following estimate:
\begin{equation}\label{2.ent-1}
\Bbb H_\eb(\Cal A_{un},E)\le D\log_2(\frac{2R_0}\eb)+\Bbb H_{\eb/K}(S_{\delta(\eb)}P_{N(\eb)}\Cal H(\mu), M_b([0,L\log_2(\frac{2R_0}\eb)], H)),
\end{equation}
where $D$, $R_0$, $K$ and $L$ are some constants and $\delta(\eb)$ and $N(\eb)$ are some functions of $\eb$ which can be explicitly found if $\mu$ is given.
\end{theorem}
\begin{proof} In order to apply Theorem \ref{Th1.main-ent}, we only need to verify estimate \eqref{1.squeez} on the uniformly absorbing ball $\Cal B\subset E$. We start with the following
key lemma which will be used in the critical case as well.
\begin{lemma}\label{Lem2.key} Let the external force $\mu$ satisfy \eqref{2.w-r}. Then, for every $\eb>0$ there exist $\delta(\eb)$ and $N(\eb)$ such that, for every $\nu\in\Cal H(\mu)$ and $\tau\in\R$,
the solution of the following linear problem:
\begin{equation}\label{2.lin}
\Dt^2 v+\gamma\Dt v-\Dx v+v=(1-\Pr(\eb))\nu,\ \ \xi_v\big|_{t=\tau}=0
\end{equation}
satisfies the estimate
\begin{equation}\label{2.small}
\|\xi_v(t)\|_E +\|v\|_{L^4(t,t+1;L^{12}(\Omega))}\le \eb,\ \ t\ge\tau.
\end{equation}
\end{lemma}
\begin {proof}[Proof of the lemma] We first note that using the obvious estimate for the solutions of linear homogeneous equation \eqref{2.lin} in $E$, we can reduce the proof of estimate \eqref{2.small} to the case where $\xi_v(t)$, $t\in\R$ is a unique bounded in time solution of \eqref{2.lin}, defined for all $t\in\R$. Thus, we get rid of the parameter $\tau$ in our problem.
\par
We now use the assumption that $\mu$ is weakly regular. Then, due to \eqref{0.good}, for any $\eb>0$   any measure $\nu\in\Cal H(\mu)$ can be presented in the form
$$
\nu=\nu_\eb+h_{time}+h_{sp},
$$
where
$$
\|\nu_\eb\|_{M_b(\R,H)}\le \eb,\ \ h_{time}\in H^1_b(\R,H),\ \ h_{sp}\in M_b^{una}(\R,H^1_0(\Omega))
$$
and this splitting is uniform with respect to $\nu\in\Cal H(\mu)$. Let $v=v_1+v_2+v_3$ be the corresponding splitting of the solution $v$ of problem \eqref{2.lin}. Then, since
$$
\|(1-\Pr(\eb))\nu_\eb\|_{M_b}\le\|\nu_\eb\|_{M_b}\le \eb,
$$
the standard energy and Strichartz estimates for \eqref{2.lin} gives us
\begin{equation}
\|\xi_{v_1}(t)\|_E+\|v_1\|_{L^4(t,t+1;L^{12})}\le C\eb,\ \ t\in\R,
\end{equation}
where $C$ is independent of $\eb$ and $\nu\in\Cal H(\mu)$. Analogously, since
$\Dt h_{time}\in L^2_b(\R,H)$ and $\Nx h_{sp}\in M_b(\R,H)$, the proper energy estimate for \eqref{2.lin} gives also that
$$
\|\xi_{v_2+v_3}\|_{E^1}\le C,
$$
where $C=C_\eb$ may depend on $\eb$, but is independent of $\nu\in\Cal H(\mu)$
and
$$
E^1:=(H^2(\Omega)\cap H^1_0(\Omega))\times H^1_0(\Omega).
$$
Thus, we may find $N=N(\eb)$ such that
\begin{equation}
\|Q_{N(\eb)}\xi_{v_2(t)+v_3(t)}\|_{E}+\|Q_{N(\eb)}(v_2+v_3)\|_{L^4(t,t+1;L^{12})}\le \eb.
\end{equation}
The Strichartz part of this estimate is also straightforward due to the embeddings
$$
H^2\subset H^{1+1/4}_0\subset L^{12}.
$$
Let us fix such $N=N(\eb)$ in the projector $\Pr(\eb)$. Then, since
\begin{equation}\label{1.pre}
1-\Pr(\eb)=1-S_\delta\circ P_N= P_N+Q_N-S_\delta\circ P_N=(1-S_\delta)\circ P_N+Q_N,
\end{equation}
we only need to verify estimate \eqref{2.small} for the solution  $\bar v$ of \eqref{2.lin} with the right-hand side  $(1-S_\delta)P_N\bar\nu$, where $\bar\nu=h_{time}+h_{sp}$. Furthermore, since both $h_{sp}$ and $h_{time}$ are weakly non-atomic, we conclude that $\bar \nu$ is also weakly non-atomic. The latter means in particular that
\begin{equation}\label{2.mod}
\lim_{\delta\to0}\sup_{t\in\R}\int_{t}^{t+\delta}\bar\nu_n(ds)=0,\ \ \bar\nu_n:=(\bar\nu,e_n).
\end{equation}
Note that $\bar v=\sum_{n=1}^N\bar v_ne_n$, where $\bar v_n$ solve scalar ODEs
\begin{equation}\label{2.ODE}
\bar v_n''(t)+\gamma \bar v_n'(t)+\lambda_n \bar v_n(t)+\bar v_n(t)=(1-S_\delta)\bar\nu_n,\ \ n=1,\cdots, N
\end{equation}
and $N=N(\eb)$ is now fixed. Thus, we only need to prove that the solution of a single ODE \eqref{2.ODE} can be made arbitrarily small by fixing $\delta=\delta(\eb)$ small enough. Note also that we actually need to verify the energy part of estimate \eqref{2.small} for $\bar v$ only (since the solution is finite dimensional, the Strichartz estimate will follow from the energy one). Let now $\tilde v_n$ be the bounded in time solution of
 \begin{equation}\label{2.lin1}
 \tilde v_n''(t)+\gamma \tilde v_n'(t)+\lambda_n \tilde v_n(t)+\tilde v_n(t)=\bar\nu_n, \ \ t\in\R.
\end{equation}
 Then, due to the energy estimate, we know that $\|\tilde v_n\|_{C^1_b(\R)}\le C$, where the constant $C$ may depend on $\eb$, but is independent of $t$. On the other hand, since $(1-S_\delta)$ is a convolution operator, it commutes with temporal derivatives and, therefore,
 $$
 \bar v_n=(1-S_\delta)\tilde v_n.
 $$
To estimate $\bar v_n$, we recall that
 \begin{multline}\label{2.sd}
((1-S_\delta)h)(t)=\frac1\delta\int_t^{t+\delta}\int_s^th'(\tau)\,d\tau\,ds=
\\
=-\frac1\delta\int_t^{t+\delta} \int^{t+\delta}_\tau h'(\tau)\,ds\,d\tau
 =-\frac1\delta\int_t^{t+\delta} (t+\delta-\tau)h'(\tau)\,d\tau.
 \end{multline}
 From this formula we immediately see that
 $$
 |(1-S_\delta)\tilde v_n(t)|\le \delta\|\tilde v_n'\|_{C_b}\le C\delta,
 $$
so this term can be easily made small by the choice of $\delta$. In order to get the analogous estimate for $\bar v_n'$, we express the second derivative $\tilde v_n''(t)$ from equation \eqref{2.lin1} and use
\eqref{2.sd} again to get
\begin{multline}\label{2.mest}
|(1-S_\delta)\tilde v_n'(t)|\le \delta(\gamma\|\tilde v_n'\|_{C_b}+(\lambda_n+1)\|\tilde v_n\|_{C_b})+\\+\frac1\delta\int_t^{t+\delta}
 \bigg|\int^{t}_s \bar\nu_n(d\tau)\bigg|\,ds\le C(\gamma+(\lambda_n+1))\delta+\sup_{|t-s|\le\delta}\big|\int_s^{t}\bar \nu_n(d\tau)\big|
\end{multline}
and to see that this term is also can be made arbitrarily small by the proper choice of $\delta=\delta(\eb)$ (due to the uniform non-atomicity \eqref{2.mod}). Thus, we are able to fix $\delta=\delta(\eb)$ in such a way that \eqref{2.small} will be satisfied and the lemma is proved.
\end{proof}

We now return to the proof of the theorem. We recall that the functions $\Pr(\eb)\nu$, $\nu\in\Cal H(\mu)$ are finite-dimensional in $x$ and even uniformly continuous in time, therefore, the projection
$\Pr(\eb)\Cal H(\mu)$ is indeed pre-compact in $M_{loc}(\R,H)$ for every $\eb>0$. Thus, Lemma \ref{Lem2.key} allows us to control the impact of the non-translation compact part of the external force $\mu$, and the rest of the proof of \eqref{1.squeez} is similar to \cite{CV,ZeDCDS}. For the convenience of the reader, we present the details below.
 \par
Let $u_1,u_2$ be two solutions of \eqref{2.wave1} starting from the absorbing set $\Cal B$ which correspond to different external forces $\nu_1,\nu_2\in\Cal H(\mu)$ and different initial data $\xi_{\tau_1},\xi_{\tau_2}\in\Cal B$ and let $w(t)=u_1(t)-u_2(t)$. Then this function solves the linear problem
\begin{equation}\label{2.w-dif}
\Dt^2w+\gamma\Dt w-\Dx w+w+l(t)w=\bar \nu:=\nu_1-\nu_2, \ \xi_{w}\big|_{t=\tau}=\bar\xi_\tau:=\xi_{\tau_1}-\xi_{\tau_2},
\end{equation}
where $l(t):=\int_0^1f'(su_1(t)+(1-s)u_2(t))\,ds$.  At first step we fix $\eb>0$ and get rid of the non-translation compact component of the external forces $\bar\nu$ by subtracting the solution $v=v_\eb(t)$ of the linear equation \eqref{2.lin} where $\nu$ is replaced by $\bar\nu$. Then the reminder $\tilde w(t)=w(t)-v(t)$ solves
\begin{equation}\label{2.w-dif2}
\Dt^2\tilde w+\gamma\Dt \tilde w-\Dx \tilde w+\tilde w+l(t)\tilde w=\tilde \nu:=\Pr(\eb)\bar \nu-l(t)v, \ \xi_{\tilde w}\big|_{t=\tau}=\bar\xi_\tau:=\xi_{\tau_1}-\xi_{\tau_2}.
\end{equation}
Then, according to Lemma \ref{Lem2.key}, $\|\xi_v\|_{E}\le \eb$. Moreover, using the dissipative estimate \eqref{2.en-str-dis} together with the growth restriction $|f'(u)|\le C(1+|u|^4)$ (the critical growth exponent is allowed here) and the Sobolev embedding $H^1\subset L^6$, we conclude that
\begin{multline}\label{2.lv}
\|lv\|_{L^1(t,t+1;H)}\le C\|l(t)\|_{L^1(t,t+1;L^3)}\|v\|_{L^\infty(t,t+1;L^6)}\le \\\le
C\(1+\|u_1\|^4_{L^4(t,t+1;L^{12})}+\|u_2\|^4_{L^4(t,t+1;L^{12})}\)\|\xi_v\|_{L^\infty(t,t+1;E)}\le C\eb
\end{multline}
for some constant $C$ which is independent of $\eb$. Thus, the term $lv$ is under the control.
\par
At second step, we utilize the standard Lipschitz continuity for $U_{\nu}(t,\tau)$ in the negative energy space $E^{-\sigma}:=H^{1-\sigma}_0(\Omega)\times H^{-\sigma}(\Omega)$ where $\sigma$ is a sufficiently small positive number. This estimate also holds for the critical growth rate as well and in order to get it, we need to multiply equation \eqref{2.w-dif2} by $(-\Dx)^{-\sigma}\Dt\tilde w$. This gives
\begin{equation}
\frac12\frac d{dt}\Cal E(\tilde w)+\gamma\|\Dt\tilde w(t)\|^2_{H^{-\sigma}}+
(l(t)\tilde w,(-\Dx)^{-\sigma}\Dt\tilde w)=(\tilde\nu,(-\Dx)^{-\sigma}\Dt\tilde w),
\end{equation}
where $\Cal E(\tilde w):=\|\xi_{\tilde w}(t)\|_{E^{-\sigma}}^2+\|\tilde w\|_{H^{-\sigma}}^{2}$.
Let us estimate the most complicated term containing $l(t)$ using the H\"older inequality with the appropriate Sobolev's embeddings. Then, analogously to \eqref{2.lv}, we have
\begin{multline}\label{2.lw}
|(l(t)\tilde w,(-\Dx)^{-\sigma}\Dt\tilde w)|\le\|l(t)\|_{L^3}\|\tilde w\|_{L^{\frac6{1+2\sigma}}}\|(-\Delta)^{-\sigma}\Dt\tilde w\|_{L^{\frac6{3-2\sigma}}}\le\\\le
C\|l(t)\|_{L^3}\|\tilde w\|_{H^{1-\sigma}}\|\Dt\tilde w\|_{H^{-\sigma}}\le C\|l(t)\|_{L^3}\|\xi_{\tilde w}\|^2_{E^{-\sigma}}.
\end{multline}
Inserting this estimate to  the previous identity, using the H\"older inequality for estimating the RHS and dividing the obtained inequality by $\Cal E^{\frac{1}{2}}(\tilde w)$, we arrive at
$$
\frac d{dt}\Cal E^{\frac{1}{2}}(\tilde w)-C\Cal E^{\frac{1}{2}}(\tilde w)\|l(t)\|_{L^3}\le \|\tilde\nu(t)\|_{H}
$$
and applying the Gronwall inequality and using \eqref{2.lv} and the fact that $\|l\|_{L^1(L^3)}$ is under the control, we get the desired estimate
\begin{equation}\label{2.lip-s}
\|\xi_{\tilde w}(t)\|_{E^{-\sigma}}\le Ce^{K(t-\tau)}\(\eb+\|\xi_{\tilde w}(\tau)\|_{E^{-\sigma}}+\|\Pr(\eb)\bar \nu\|_{L_b^1(\tau,t;H)}\),
\end{equation}
for some constants $K$ and $C$ which are independent of $\eb$, $t$,$\tau$ and $u_1,u_2$.
Note that this estimate remains valid for the critical case as well. We also mention that the measure $\Pr(\eb)\bar\nu$ is absolutely continuous, so we may write $L^1_b(H)$ instead of $M_b(H)$.
\par
Finally, at step 3 we write out the energy estimate in $E$ for equation \eqref{2.w-dif2} treating the term $l(t)\tilde w$ as a perturbation. In contrast to the previous estimates, we will crucially use here that the non-linearity $f$ has a subcritical growth rate (in the next section, we remove these restrictions using more delicate arguments). The sub-criticality assumption $p<5$ allows us to improve estimate \eqref{2.lv} as follows:
\begin{multline}\label{2.subcrit}
\|l(t)\tilde w\|_{L^2}\le \|l(t)\|_{L^{\frac{12}{p-1}}}\|\tilde w\|_{L^{\frac{12}{7-p}}}\le C\(1+\|u_1(t)\|_{L^{12}}^4+\|u_2(t)\|_{L^{12}}^4\)\|\tilde w\|_{H^{1-\sigma}}\le\\\le
C\(1+\|u_1(t)\|_{L^{12}}^4+\|u_2(t)\|_{L^{12}}^4\)\|\xi_{\tilde w}\|_{E^{-\sigma}},
\end{multline}
where $\sigma=\sigma(p)>0$ due to the sub-criticality assumption.
\par
Then, multiplying equation \eqref{2.w-dif2} by $\Dt\tilde w+\alpha\tilde w$ and arguing in a standard way, we get the inequality
\begin{multline}\label{2.gron-m}
\frac12\frac d{dt}\tilde {\Cal E}(\tilde w)+\beta\tilde {\Cal E}(\tilde w)\le C\(1+\|u_1(t)\|_{L^{12}}^4+\|u_2(t)\|_{L^{12}}^4\)\|\xi_{\tilde w}\|_{E^{-\sigma}}\tilde {\Cal E}^{\frac{1}{2}}(\tilde w)+C\|\tilde\nu(t)\|_{L^2}\tilde {\Cal E}^{\frac{1}{2}}(\tilde w)
\end{multline}
for some positive constant $\beta$. Here
$$
\tilde {\Cal E}(w):=\|\Dt w\|_{L^{2}}^{2}+\|\Nx w\|_{L^{2}}^{2}+\|w\|_{L^{2}}^{2}+\gamma\alpha\| w\|_{L^{2}}^{2}+2\alpha(\Dt w,w).
$$
For sufficiently small $\alpha>0$, we have
$$
\frac{1}{2}\|\xi_{\tilde w}(t)\|_{E}^{2}\le\tilde {\Cal E}(\tilde w).
$$
Applying the Gronwall inequality to this relation and using \eqref{2.lip-s} together with \eqref{2.lv} and \eqref{2.small}, we finally arrive at
\begin{equation}\label{2.squeez}
\|\xi_w(t)\|_{E}\le e^{-\beta (t-\tau)}\|\xi_{w}(\tau)\|_{E}+Ce^{K(t-\tau)}\(\eb+\|\xi_w(\tau)\|_{E^{-\sigma}}+
\|\Pr(\eb)\bar\nu\|_{L_b^1(\tau,t;H)}\)
\end{equation}
for some positive constants $\beta$, $K$ and $C$ which are independent of $\eb$, $t$, $\tau$ and $u_1,u_2$. Since the embedding $E\subset E^{-\sigma}$ is compact, estimate \eqref{2.squeez} gives us the desired estimate \eqref{1.squeez} with $W=E^{-\sigma}$ (up to the proper scaling of $\eb$ if necessary). The entropy estimate \eqref{2.ent-1} is then follows from Theorem \ref{Th1.main-ent} and the theorem is proved.
 \end{proof}
\begin{remark} In the case when the external force $\mu$ is translation-compact in $L^1_b(\R,H)$, the projector $\Pr(\eb)$ can be replaced by $\operatorname{Id}$ and we end up with the standard formula for Kolmogorov's entropy of Vishik and Chepyzhov, \cite{CV}. Analogously, if $\mu$ is space (time) regular, we may omit the corresponding smoothing and replace $\Pr(\eb)$ by $S_{\delta(\eb)}$   ($P_{N(\eb)}$) respectively. We also note  that as it follows from the proof of Lemma \ref{Lem2.key} the concrete form of time and space smoothing operators ($S_{\delta(\eb)}$ and $P_{N(\eb)}$) is also not important and they can be replaced by more general smoothing/convolution operators, we just fix them in the form which is most convenient for us.
\end{remark}

\section{Application 2. Critical damped wave equation with periodic BC}\label{s.5}

In this section, we obtain the analogue of Theorem \ref{Th2.main-ent} for the quintic weakly damped wave equation \eqref{2.wave}. As in the previous section, we will assume that our external force $\mu$ is (weakly) uniformly non-atomic and is weakly regular  (i.e. \eqref{2.w-r} is satisfied), but
we now consider the case of periodic boundary conditions only and
the nonlinearity $f$ is assumed to satisfy
\begin{equation}\label{3.f}
1. \ f(u)=u^5+h(u),\ \ 2.\  h\in C^2(\R),\ h(0)=0,\ \ 3.\  |h''(u)|\le C(1+|u|^q),\ \ q<3
\end{equation}
in order to be consistent with \cite{SZUMN}.
\par
 The definition of the SS-solution of \eqref{2.wave} in the critical quintic case is exactly the same as for the subcritical case and the analogues of Theorem \ref{Th2.wp} and Corollary \ref{Cor2.weak} are proved in \cite{SZUMN}, so we will concentrate here on obtaining the analogues of Theorem \ref{Th2.main-ent}. Namely, we intend to prove the following result.

\begin{theorem}\label{Th3.main-ent} Let the nonlinearity $f$ and the external measure $\mu$ satisfy \eqref{3.f} and \eqref{2.w-r} respectively and let, in addition, the wave equation \eqref{2.wave} be endowed with periodic boundary conditions. Then the family of processes $\{U_\nu(t,\tau),\nu\in\Cal H(\mu)\}$ generated by equations \eqref{2.wave1} possess a strong uniform attractor $\Cal A_{un}=\Cal A^s_{un}=\Cal A_{un}^w$ in the energy phase space $E$. Moreover, its Kolmogorov's entropy possesses the following estimate:
\begin{equation}\label{3.ent-1}
\Bbb H_\eb(\Cal A_{un},E)\le D\log_2(\frac{2R_0}\eb)+\Bbb H_{\eb/K}(S_{\delta(\eb)}P_{N(\eb)}\Cal H(\mu), M_b([0,L\log_2(\frac{2R_0}\eb)], H)),
\end{equation}
where $D$, $R_0$, $K$ and $L$ are some constants and $\delta(\eb)$ and $N(\eb)$ are some functions of $\eb$ which can be explicitly found if $\mu$ is given.
\end{theorem}
Note from the very beginning that the periodic boundary conditions are necessary for the energy-to-Strichartz estimate \eqref{2.en-str} only and is not used anywhere else. The strategy of the proof is the same as in the subcritical case, the only difference is that  we do  not have estimate \eqref{2.subcrit} in the critical case, so we need to replace this estimate by something more delicate. To this end, we need the following result which is the key technical tool for our proof.

\begin{proposition}\label{Prop3.split} Let the assumptions of Theorem \ref{Th3.main-ent} hold. Then, for every $\eb>0$ there exists $T=T(\eb)$ such that every solution  $u(t)$ of \eqref{2.wave1} starting from the unform absorbing set $\Cal B$ (i.e. $\xi_u\big|_{t=\tau}=\xi_\tau\in \Cal B$) can be split in two parts $u(t)=u_s(t)+u_c(t)$ such that, for every $t\ge\tau+T(\eb)$,
\begin{equation}\label{3.ssplit}
\|\xi_{u_s}(t)\|_{ E}+\|u_s\|_{L^4(t,t+1;L^{12})}\le\eb,\ \ \|\xi_{u_c}(t)\|_{E^\sigma}+\|u_c\|_{L^4(t,t+1;W^{\sigma,12})}\le C_\eb,
\end{equation}
where the constant $C_\eb$ depends on $\eb$, but is independent of $t$, $\tau$, $\nu\in\Cal H(\mu)$ and $\xi_\tau\in\Cal B$ and $\sigma>0$ is small enough.
\end{proposition}
\begin{proof} Following \cite{SZUMN}, we first construct an auxiliary split of $u(t)=u_1(t)+u_2(t)$ on a small and smooth, but exponentially growing components. Namely, we set $\tau=0$ for simplicity and  take $u_1(t)$ as a solution of
\begin{equation}\label{3.small}
\Dt^2u_1+\gamma \Dt u_1-\Dx u_1+u_1+f(u_1)+L u_1=(1-\Pr(\eb))\nu,\ \ \xi_{u_1}\big|_{t=0}=\xi_u\big|_{t=0},
\end{equation}
where $L=L_f$ is a sufficiently large constant which is introduced to compensate the instability of $f$ and $\eb$ is a sufficiently small positive number. Then the reminder obviously solves
\begin{equation}\label{3.exp}
\Dt^2 u_2+\gamma\Dt u_2-\Dx u_2+u_2+[f(u_1+u_2)-f(u_1)]=Lu_1+\Pr(\eb)\nu,\  \ \xi_{u_2}\big|_{t=0}=0.
\end{equation}
We start with the $u_1$-component and prove the following result.
\begin{lemma}\label{Lem3.small} Let the above assumptions hold. Then the function $u_1$ satisfies the following estimate:
\begin{equation}\label{3.u1sm}
\|\xi_{u_1}(t)\|_{E}+\|u_1\|_{L^4(t,t+1;L^{12})}\le C(\eb+Ke^{-\alpha t}),
\end{equation}
where the positive constants $C$, $K$ and $\alpha$ are independent of $\eb$, $\nu$ and $u$.
\end{lemma}
\begin{proof}[Proof of the Lemma]  We split the proof on several steps.
\par
{\it Step 1.} The function $u_1$ is uniformly bounded in time.  Indeed, this is an immediate corollary of the dissipative estimate \eqref{2.en-str-dis} applied to equation \eqref{3.small} (the presence of the extra term $Lu_1$ is not essential and can be treated as the part of the non-linearity $f$). Thus, we have the estimate
 \begin{equation}\label{3.bound}
 \|\xi_{u_1}(t)\|_E+\|u_1\|_{L^4(t,t+1;L^{12})}\le C,
 \end{equation}
 where $C$ is independent of $t$, $\nu$, and $u$.
 \par
 {\it Step 2.} The function $u_1$ is asymptotically small in the energy norm. Let $\bar u_1(t)$ be a solution of the linear problem
 $$
 \Dt\bar u_1+\gamma\Dt \bar u_1-\Dx\bar u_1+\bar u_1+L\bar u_1=(1-\Pr(\eb))\nu,\ \ \xi_{\bar u_1}\big|_{t=0}=0.
 $$
 Then, according to Lemma \ref{Lem2.key}, we have the estimate
 \begin{equation}\label{3.lin-small}
 \|\xi_{\bar u_1}\|_E+\|\bar u_1\|_{L^4(t,t+1;L^{12})}\le \eb
 \end{equation}
(the presence of an extra term $L\bar u_1$ does not change anything in the proof). Let $\tilde u_1:=u_1-\bar u_1$. Then, this function solves
\begin{equation}\label{3.tilde1}
\Dt^2\tilde u_1+\gamma\Dt\tilde u_1-\Dx\tilde u_1+\tilde u_1+f(\tilde u_1+\bar u_1)+L\tilde u_1=0,\ \xi_{\tilde u_1}\big|_{t=0}=\xi_0.
\end{equation}
 We multiply this equation by $\Dt\tilde u_1+\alpha\tilde u_1$. Then, after the standard transformations, we arrive at
 \begin{multline}
 \frac12\frac d{dt}\(\|\Dt \tilde u_1\|^2_{L^{2}}+\|\Nx \tilde u_1\|_{L^2}^2+(L+1)\|\tilde u_1\|_{L^2}^2
 +2(F(\tilde u_1),1)+\alpha\gamma\|\tilde u_1\|_{L^2}^{2}+2\alpha(\Dt\tilde u_1,\tilde u_1)\)+\\+
 (\gamma-\alpha)\|\Dt\tilde u_1\|^2+\alpha\|\Nx\tilde u_1\|^2_{L^2}+\alpha (L+1)\|\tilde u_1\|^2_{L^2}+\alpha(f(\tilde u_1),\tilde u_1)=\\=(f(\tilde u_1)-f(\bar u_1+\tilde u_1),\Dt\tilde u_1+\alpha\tilde u_1).
\end{multline}
  Due to our conditions on the nonlinearity $f$, we may fix the constant $L$ in such a way that the function $f_L(u):=f(u)+Lu$ satisfies
$$
1. \ F_L(u)\ge|u|^2,\ \ 2.\  f_L(u)u\ge|u|^2,\ \ 3.\  F_L(u)\le\kappa f_L(u)u+C|u|^2
$$
for some positive $\kappa$ and $C$. Then, performing the standard estimates and fixing $\alpha$ to be small enough,  we end up with the following differential inequality:
\begin{equation}\label{3.gron}
\frac12\frac d{dt}\Cal E(\tilde u_1)+\beta\Cal E(\tilde u_1)\le \|f(\tilde u_1(t)+\bar u_1(t))-f(\tilde u_1(t))\|_{L^2}[\Cal E(\tilde u_1)]^{1/2},
\end{equation}
where
$$
\Cal E(u)=\|\Dt u\|^2_{L^2}+\|\Nx u\|^2_{L^2}+\|u\|^2_{L^2}+2(F_L(u),1)+\gamma\alpha\|u\|^2_{L^2}+2\alpha(\Dt u,u),
$$
see \cite{SZUMN} for more details. In order to estimate the right-hand side of \eqref{3.gron}, we use estimates \eqref{3.bound} and \eqref{3.lin-small} together with the fact that $f'(u)$ has the quartic growth rate, namely,
\begin{multline}
\|f(\tilde u_1(t)+\bar u_1(t))-f(\tilde u_1(t))\|_{L^1(t,t+1,L^2)}\le\\\le
C\(1+\|\tilde u_1\|_{L^4(L^{12})}^4+\|u_1\|_{L^4(L^{12})}^4\)\|\bar u_1\|_{L^\infty(L^{6})}\le C\|\xi_{\bar u_1}\|_{L^\infty(t,t+1;E)}\le C\eb.
\end{multline}
It remains to note that, for sufficiently small $\alpha>0$, we have
$$
\frac12 \|\xi_u\|_E^{2}\le \Cal E(u)
$$
and, therefore, the Gronwall inequality applied to \eqref{3.gron} gives us the desired asymptotic smallness in the energy norm:
\begin{equation}\label{3.sm-en}
\|\xi_{\tilde u_1}(t)\|_{E}\le C(\eb+K e^{-\beta t}),
\end{equation}
where $C$, $K$ and $\alpha$ are independent of $\eb$, $\nu$, $t$ and $u$.
\par
{\it Step 3.} The function $ u_1$ is asymptotically small in the Strichartz norm. We return to equation \eqref{3.tilde1} and apply the energy-to-Strichartz estimate treating the term $f(\tilde u_1+\bar u_1)$ as a perturbation. This gives
$$
\|\tilde u_1\|_{L^4(t,t+1;L^{12})}\le C\(\|\xi_{\tilde u_1}(t)\|_{E}+\|f(\tilde u_1+\bar u_1)\|_{L^1(t,t+1;L^{2})}\).
$$
Using the fact that $f'$ has the quartic growth rate and that $f(0)=0$ together with estimates \eqref{3.lin-small} and \eqref{3.bound}, we get
\begin{multline}
\|f(\tilde u_1+\bar u_1)\|_{L^1(t,t+1;L^{2})}\le C(1+\|\tilde u_1\|^4_{L^4(L^{12})}+\|\bar u_1\|^4_{L^4(L^{12})})(\|\xi_{\tilde u_1}\|_{L^\infty(E)}+\|\xi_{\bar u_1}\|_{L^\infty(E)})\le\\\le C(\eb+\|\xi_{\tilde u_1}\|_{L^{\infty}(t,t+1;E)}).
\end{multline}
Thus,
$$
\|\tilde u_1\|_{L^4(t,t+1;L^{12})}\le C(\eb+\|\xi_{\tilde u_1}\|_{L^\infty(t,t+1;E)})
$$
and this estimate together with \eqref{3.sm-en} and \eqref{3.lin-small} give the desired asymptotic
smallness of the Strichartz norm of $u_1$ which in turn gives the desired estimate \eqref{3.u1sm} and finishes the proof of the lemma.
\end{proof}
We now turn to the $u_2$-component of the solution $u$ which is estimated in the next lemma.
\begin{lemma} Let the above assumptions hold and let $\sigma\in(0,\frac25)$. Then, the following estimate holds:
\begin{equation}\label{3.exp1}
\|\xi_{u_2}(t)\|_{E^{\sigma}}+\|u_2\|_{L^4(t,t+1;W^{\sigma,12})}\le C e^{Kt},
\end{equation}
where the constants $C$ and $K$ may depend on $\eb$, but are independent of $u$, $\nu\in\Cal H(\mu)$ and $t$.
\end{lemma}
The proof of this lemma is identical to \cite[Lemma 8.5]{SZUMN} and so is omitted. We only note here that the term $\Pr(\eb)\nu$ is smooth in space and time, so it cannot produce any difficulties except of making the constants $C$ and $K$ depending on $\eb$.
\par
We are now ready to finish the proof of the proposition.
Let us now fix a big number $T=T(\eb)$ and
  consider decompositions $u(t)=u_1^n(t)+u_2^n(t)$, $t\ge nT$ which are defined by equations \eqref{3.small}
   and \eqref{3.exp}, but starting with the time moment $T_n=T(n-1)$ with the initial data
   $$
    \xi_{u_1^n}\big|_{t=T_n}=\xi_u\big|_{t=T_n},\ \ \xi_{u_2^n}\big|_{t=T_n}=0.
   $$
   Then, due to estimates \eqref{3.u1sm} and \eqref{3.exp1} we get
   \begin{equation}\label{3.1n}
   \|\xi_{u^n_1}(t)\|_{\E}+\|u_1^n\|_{L^4(t,t+1;L^{12})}\le C(\eb+Ke^{-\beta t})\le2C\eb
   \end{equation}
   if $t\ge T_n+T=Tn$ and $T$ is chosen to satisfy
    $$
    CKe^{-\beta T}=C\eb
    $$
    and
   \begin{equation}\label{3.2n}
\|\xi_{u_2^n}(t)\|_{\E}+\|u_2^n\|_{L^4(t,t+1;L^{12})}\le Ce^{Kt}\le M_\eb
   \end{equation}
   if $t\le T_n+2T=T(n+1)$.
\par
Finally, we define the desired functions $u_s(t)$ and $u_c(t)$ for $t\ge T_\eb:=T$ as piece-wise
continuous hybrid functions:
\begin{equation}\label{hybrid}
u_s(t):=u^n_1(t),\ \ u_c(t):=u^n_2(t),\ \ t\in [nT,(n+1)T).
\end{equation}
The desired properties of $u_s$ and $u_c$ are now guaranteed by estimates \eqref{3.1n} and \eqref{3.2n}. Crucial for the above construction is that we define the functions $u_1^n(t)$ and $u_2^n(t)$
 starting from the initial time $T_n=T(n-1)$, but are in \eqref{hybrid} on the time
  interval $t\in[Tn,T(n+1)]$ only and this time shift of length $T$ guarantees that $u_1^n(t)$ is already
  small for	$t\in[Tn,T(n+1)]$ due to exponential decay of the $u_1^n$-component. Thus, the proposition is proved. 	
\end{proof}
We are now ready to prove the main result of this section.
\begin{proof}[Proof of Theorem \ref{Th3.main-ent}] As we have already mentioned, the proof of this theorem follows the scheme presented in the proof of Theorem \ref{Th2.main-ent} for the subcritical case. The only estimate which fails in the critical case is \eqref{2.subcrit} which now holds for $\sigma=0$ only, so we only need to modify it properly using Proposition \ref{Prop3.split}. Indeed, let
as in the proof of Theorem \ref{Th2.main-ent}, $u_1(t)$ and $u_2(t)$ be two solutions of \eqref{2.wave1} with external forces $\nu_1$, $\nu_2\in\Cal H(\mu)$ and different initial data starting from the absorbing ball $\Cal B$ at
$t=\tau$. Then, we take a sufficiently small $\alpha>0$ which will be specified below and which is {\it independent} of $\eb$ and split every of two solutions $u_1$ and $u_2$ on small and smooth parts:
$$
u_1(t)=u_{1,s}(t)+u_{1,c}(t),\ \ u_2(t)=u_{2,s}(t)+u_{2,c}(t)
$$
such that  inequalities \eqref{3.ssplit} hold with $\alpha$ instead of $\eb$ for all $t\ge \tau+T(\alpha)$. Important that $T(\alpha)$ and $C_\alpha$ are independent of $u_1,u_2$, $\nu_1$, $\nu_2$, $\eb$ and $t$. Moreover, replacing the absorbing set $\Cal B$ by the new absorbing set
$$
\Cal B_\alpha:=\cup_{\nu\in\Cal H(\mu)}U_\nu(T(\alpha),0)\Cal B,
$$
if necessary, we may assume without loss of generality that \eqref{3.ssplit} hold for all $t\ge\tau$. Using this splitting, we may write
\begin{multline}
l(t)=\int_0^1f'(s u_1(t)+(1-s)u_2(t))\,ds=\\=
\int_0^1[f'(su_1(t)+(1-s)u_2(t))-f'(su_{1,c}(t)+(1-s)u_{2,c}(t))]\,ds+\\+
\int_0^1f'(su_{1,c}(t)+(1-s)u_{2,c}(t))\,ds:=l_s(t)+l_c(t).
\end{multline}
 Since $u_{i,c}(t)$ are more regular,  arguing as in the proof of \eqref{2.subcrit}, we get
\begin{equation*}
\begin{split}
&\|l_c(t)\tilde w(t)\|_{L^2}\le C\|l_c(t)\|_{L^{\frac{3}{1-4\sigma}}}\|\tilde w(t)\|_{L^{\frac{6}{1+8\sigma}}}\\
&\ \ \ \ \ \ \ \ \ \ \ \ \ \ \ \ \ \le C(\|u_{1,c}\|_{L^{\frac{12}{1-4\sigma}}}^{4}+\|u_{2,c}\|_{L^{\frac{12}{1-4\sigma}}}^{4})\|\tilde w(t)\|_{L^{\frac{6}{1+2\sigma}}}\\
&\ \ \ \ \ \ \ \ \ \ \ \ \ \ \ \ \ \le C(\|u_{1,c}\|_{W^{\sigma,12}}^{4}+\|u_{2,c}\|_{W^{\sigma,12}}^{4})\|\tilde w(t)\|_{H^{1-\sigma}}:=m_c(t)\|\xi_{\tilde w}\|_{E^{-\sigma}},
\end{split}
\end{equation*}
where the  scalar function $m_c(t)\ge0$  satisfies
\begin{equation}\label{3.m-bound}
\int_t^{t+1}m_c(s)\,ds\le C_\alpha,
\end{equation}
where $C_\alpha$ is independent of $u_i$, $\nu_i$, $\eb$ and $t$. On the other hand,
\begin{multline}
l_s(t)=\int_0^1\int_0^1f''(s_1(su_1(t)+(1-s)u_2(t))+\\+(1-s_1)(su_{1,c}(t)+(1-s)u_{2,c}(t)))\,ds_1
(su_{1,s}(t)+(1-s)u_{2,s}(t))\,ds
\end{multline}
and therefore, due to our assumptions on the nonlinearity $f$,
\begin{equation}
\|l_s(t)\|_{L^3}\le C\(1+\sum_{i=1}^2(\|u_{i,c}(t)\|^3_{L^{12}}+\|u_{i,s}(t)\|^3_{L^{12}})\)
\(\sum_{i=1}^2\|u_{i,s}(t)\|_{L^{12}}\).
\end{equation}
Thus, due to the H\"older inequality, we have
\begin{equation}
\|l_s(t)\tilde w\|_{L^2}\le m_s(t)\|\xi_{\tilde w}(t)\|_E,
\end{equation}
where the scalar function $m_s(t)\ge0$ satisfies
\begin{equation}\label{3.m-small}
\int_t^{t+1}m_s(\kappa)\,d\kappa\le C\alpha,
\end{equation}
for some $C$ which is independent of $\alpha$, $\eb$, $u_i$, $t$ and $\nu_i$. Thus, \eqref{2.subcrit} is finally replaced by
\begin{equation}
\|l(t)\tilde w(t)\|_{L^2}\le m_s(t)\|\xi_{\tilde w}(t)\|_E+m_c(t)\|\xi_{\tilde w}(t)\|_{E^{-\sigma}}
\end{equation}
and the analogue of \eqref{2.gron-m} will now read
\begin{equation}\label{3.gron-m}
\frac d{dt}\tilde{\Cal E}^{\frac{1}{2}}(\tilde w)+(\beta-m_s(t))\tilde{\Cal E}^{\frac{1}{2}}(\tilde w)\le m_c(t)\|\xi_{\tilde w}\|_{E^{-\sigma}}+\|\tilde\nu(t)\|_{L^2}.
\end{equation}
Fixing now $\alpha>0$ in such a way that the constant $C$ in \eqref{3.m-small} satisfies $C\alpha\le\frac\beta2$ and applying the Gronwall inequality, we end up with the desired estimate \eqref{2.squeez} and finish the proof of the theorem.

\end{proof}

\section{Application 3. 2D Navier-Stokes system}\label{s6}

In this section, we consider the following  Navier-Stokes system:
\begin{equation}\label{4.NS}
\Dt u+(u,\Nx)u+\Nx p=\Dx u+g(t),\ \ \divv u=0,\ \ u\big|_{\partial\Omega}=0, \ u\big|_{t=\tau}=u_\tau
\end{equation}
in a bounded smooth domain $\Omega$ of $\R^2$. Here $u=(u_1,u_2)$ and $p$ are unknown velocity field and pressure respectively and $g$ is a given external force. We introduce as usual the space $\Cal D_\sigma(\Omega)$ as the subspace of $C_0^\infty(\Omega)$ which consists of solenoidal vector fields as well as the corresponding divergence free distributions $\Cal D'_\sigma(\Omega)$. The energy phase space for problem \eqref{4.NS} is defined as follows:
$$
H:=[\Cal D_\sigma(\Omega)]_{[L^2(\Omega)]^2}.
$$
Let $\Pi:[L^2(\Omega)]^2\to H$ be the Leray orthoprojector and $A:=-\Pi\Dx$ be the Stokes operator. We also introduce the fractional powers $A^s$, $s\in\R$ of the Stokes operator and the corresponding spaces $H^s=H^s_A:=D(A^{s/2})$. Then, as known
$$
H^s_A=D((-\Dx)^{s/2})\cap\{u.n\big|_{\partial\Omega}=0\},\ \ \frac12\le s\le2,
$$
where $(-\Dx)$ is the Laplacian endowed with Dirichlet boundary conditions. In particular,
$$
H=\{u\in [L^2(\Omega)]^2,\ \ \divv u=0,\ \ u.n\big|_{\partial\Omega}=0\},\  V:=H^{1}=\{u\in[H^1_0(\Omega)]^2,\ \divv u=0\},
$$
see \cite{Temam} and references therein for more details.
\par
Throughout  this section, we assume that
\begin{equation}\label{4.g}
g\in L^2_b(\R,V')
\end{equation}
and for strong attractors and entropy estimates we will assume also that $g$ is weakly normal in this space.
\par
We recall that the function $u(t)$ is a weak-solution of problem \eqref{4.NS} if, for every $T>\tau$,
\begin{equation}\label{4.def-sol}
u\in C([\tau,T],H)\cap L^2(\tau,T;V)
\end{equation}
and equation \eqref{4.NS} is satisfied in the sense of divergence-free distributions.
\par
The following result is classical and can be found e.g. in \cite{BV,CV,Temam}.

\begin{theorem}\label{Th4.dis} Let the external force $g$ satisfy \eqref{4.g} and $u_\tau\in H$. Then problem \eqref{4.NS} possesses a unique (weak) solution $u(t)$ and this solution satisfies the energy identity
$$
\frac12\frac d{dt}\|u(t)\|^2_H+\|\Nx u(t)\|^2_{L^2}=(g(t),u(t)), \text{ for almost all $t\ge\tau$ }
$$
as well as the following dissipative estimate:
\begin{equation}\label{4.dis}
\|u(t)\|_H+\|u\|_{L^2(t,t+1;V)}\le Ce^{-\beta (t-\tau)}\|u_\tau\|_H+C\|g\|_{L^2_b(\R,V')},
\end{equation}
where the positive constants $C$ and $\beta$ are independent of $u$, $g$, $t$ and $\tau$.
\end{theorem}
We now turn to attractors. Following the general scheme, we need to consider the family of Navier-Stokes equations
\begin{equation}\label{4.NS1}
\Dt u+(u,\Nx)u+\Nx p=\Dx u+\nu(t),\ \ \divv u=0,\ \ u\big|_{\partial\Omega}=0, \ u\big|_{t=\tau}=u_\tau,
\end{equation}
where the external force $\nu\in \Cal H(g)$ and define a family of dynamical processes $\{U_\nu(t,\tau), \nu\in \Cal H(g)\}$ generated by the solution operators of these problems. The well-posedness and uniform dissipativity of this family is guaranteed by Theorem \ref{Th4.dis}. Note that, in contrast to the previous sections the weak hull $\Cal H(g)$ does not contain singular measures, but consists of functions $\nu\in L^2_b(\R,V')$. Moreover, the weak continuity of the maps $(u_\tau,\nu)\to U_\nu(t,\tau)u_\tau$ for every fixed $t\ge\tau$ is also straightforward, so due to Theorem \ref{Th1.exist} and Corollary \ref{Cor1.rep}, we have the following result, see \cite{CV} for more details.
\begin{corollary} Let the assumptions of Theorem \ref{Th4.dis} hold. Then the family of dynamical processes $\{U_\nu(t,\tau), \nu\in\Cal H(g)\}$ possesses a weak uniform attractor $\Cal A^w_{un}$ in the phase space $H$ and the following representation formula holds:
$$
\Cal A_{un}^w=\cup_{\nu\in\Cal H(g)}\Cal K_{\nu}\big|_{t=0},
$$
where $\Cal K_\nu\in C_b(\R,H)$ is a set of all bounded solutions of \eqref{4.NS1} defined for all $t\in\R$.
\end{corollary}
The next result which gives the existence of a strong uniform attractor and the upper bounds for its entropy can be considered as a main result of this section.
\begin{theorem}\label{Th4.main-ent} Let the assumptions of Theorem \ref{Th4.dis} hold and let, in addition, the initial external force $g$ be weakly normal in $L^2_b(\R,V')$. Then the family of processes $\{U_\nu(t,\tau),\nu\in\Cal H(g)\}$ generated by equations \eqref{4.NS1} possess a strong uniform attractor $\Cal A_{un}=\Cal A^s_{un}=\Cal A_{un}^w$ in $H$. Moreover, its Kolmogorov's entropy possesses the following estimate:
\begin{equation}\label{4.ent-1}
\Bbb H_\eb(\Cal A_{un},H)\le D\log_2(\frac{2R_0}\eb)+\Bbb H_{\eb/K}(S_{\delta(\eb)}P_{N(\eb)}\Cal H(g), L^2([0,L\log_2(\frac{2R_0}\eb)], V')),
\end{equation}
where $D$, $R_0$, $K$ and $L$ are some constants and $\delta(\eb)$ and $N(\eb)$ are some functions of $\eb$ which can be explicitly found if $g$ is given. Here and below, $\Pr(\eb)$ is defined by \eqref{2.pr} and $P_N$ and $Q_N$ are orthoprojectors related with the orthonormal base in $H$ generated by the eigenvectors of the Stokes operator.
\end{theorem}
As in the previous sections, we need the analogue of Lemma \ref{Lem2.key} in order to treat the non-translation compact part of the external forces $g$.
\begin{lemma}\label{Lem4.key} Let the function $g\in L^2_b(\R,V')$ be weakly normal. Then, for every $\eb>0$, there exist $\delta=\delta(\eb)$ and $N=N(\eb)$ such that, for every $\nu\in\Cal H(g)$
 and every $\tau\in\R$, the solution $v(t)$ of the linear problem
 \begin{equation}\label{4.lin}
 \Dt v+Av=(1-\Pr(\eb))\nu(t),\ \ v\big|_{t=\tau}=0
 \end{equation}
 satisfies the estimate
 \begin{equation}\label{4.lin-small}
 \|v(t)\|_{H}\le \eb,\ \ \ \|v\|_{L^2(t,t+1;V)}\le C,\ \ \ t\ge\tau,
 \end{equation}
 where the constant $C$ is independent of $\eb$.
\end{lemma}
\begin{proof} We first note that without loss of generality, we may assume that $\tau=0$. Moreover,  since $g$ is weakly normal, due to \eqref{0.exp-w-norm1}, for every $\eb>0$, there are exist $N(\eb)$ and $L(\eb)$ such that
\begin{equation}\label{4.exp}
\int_{-\infty}^te^{-L(t-s)}\|Q_N\nu(s)\|^2_{V'}\,ds\le \eb,\  L\ge L(\eb), \ N\ge N(\eb),\ t\in\R,\ \ \nu\in\Cal H(g).
\end{equation}
Multiplying equation \eqref{4.lin} by $Q_Nv$ and using that $Q_N\Pr(\eb)=0$, after applying the Gronwall inequality, we get the standard estimate
$$
\|Q_Nv(t)\|_H^2\le \int_0^t e^{-\lambda_N(t-s)}\|Q_N\nu(s)\|^2_{H^{-1}}\,ds\le \int_{-\infty}^t e^{-\lambda_N(t-s)}\|Q_N\nu(s)\|^2_{H^{-1}}\,ds.
$$
Thus, using the monotonicity of the left-hand side of \eqref{4.exp} in $N$ and $L$,
 we may fix $N=N(\eb)$ in such a way that
 \begin{equation}
 \|Q_{N(\eb)}v(t)\|_{H}\le \frac\eb2,\ \ t\ge0.
 \end{equation}
 Therefore, according to \eqref{1.pre}, it only remains to estimate the $P_N$-component of the solution $v(t)$ of  \eqref{4.lin} where we may write
 $(1-S_\delta)P_N \nu$ in the right-hand side.  To this end, we note that the function $g$ is  weakly and (even strongly) non-atomic,
since by the H\"older inequality
$$
\|\int_t^{t+\delta}g(s)\,ds\|_{V'}\le\delta^{1/2}\|g\|_{L^2_b(\R,V')},\ \ t\in\R,\ \ \delta\le 1
$$
and, therefore,
$$
\|\int_t^{t+\delta}P_Ng(s)\,ds\|_{H}\le C_N\delta^{1/2}\|g\|_{L^2_b(\R,V')},\ \ t\in\R,\ \ \delta\le 1,
$$
where the constant $C_N$ is independent of $\delta$.
\par
Recall that $P_Nv(t)$ solves a system of $N$ first order ODEs for the Fourier components and $N=N(\eb)$ is already fixed, so, arguing exactly as in the proof of Lemma \ref{Lem2.key}, we may finally fix $\delta=\delta(\eb)$ in such a way that
$$
\|P_Nv(t)\|_{H}\le\frac\eb2
$$
and finish the proof of the first estimate of \eqref{4.lin-small}. The second one follows immediately from the fact that the $L^2_b(V')$-norm of $(1-\Pr(\eb))\nu$ is majored by the analogous norm of $\nu$ uniformly with respect to $\eb$. This finishes the proof of the lemma.
\end{proof}
We are now ready to finish the proof of the key theorem.

\begin{proof}[Proof of Theorem \ref{Th4.main-ent}] Indeed, let $u_1(t)$ and $u_2(t)$ be two solutions of \eqref{4.NS1} with the external forces $\nu_1,\nu_2\in\Cal H(g)$ starting at $t=0$ from the uniformly absorbing set $\Cal B$ (which exists due to the dissipative estimate \eqref{4.dis}). Then, the difference $\bar u(t):=u_1(t)-u_2(t)$ solves the linear equation
\begin{equation}\label{4.linNS}
\Dt\bar u+(u_1,\Nx)\bar u+(\bar u,\Nx)u_2+\Nx\bar p=\Dx\bar u +\bar \nu(t),\ \ \divv\bar u=0, \ \bar u\big|_{\partial\Omega}=0,
\end{equation}
where $\bar\nu(t)=\nu_1(t)-\nu_2(t)$. As before, we use Lemma \ref{Lem4.key} in order to get rid of the non-translation compact component of $\bar \nu(t)$ in this equation. Namely, $\eb>0$ be fixed and let $v(t)$ be a solution of equation \eqref{4.lin} with the right-hand side $(1-\Pr(\eb))\bar\nu(t)$. Then the function $\tilde u(t):=\bar u(t)-v(t)$ solves
\begin{equation}\label{4.NS-dif}
\Dt \tilde u+(u_1,\Nx) \tilde u+(\tilde u, \Nx)u_2+\Nx\tilde p=\Dx\tilde u-(u_1,\Nx)v-(v,\Nx)u_2+\Pr(\eb)\bar \nu.
\end{equation}
In order to estimate the nonlinear terms in this relation, we need the following estimates.
\begin{lemma}Let $b(u,v,w):=((u,\Nx v),w)$ be the three-linear form associated with the inertial term. Then, for every $u,v,w\in H^1$ and every $0\le\sigma\le1/2$, the following estimates hold:
\begin{equation}\label{4.in}
|b(u,v,w)|+|b(v,u,w)|\le C\|u\|_{L^4(\Omega)}\|v\|_{H^{-\sigma}}^{1/2}\|v\|_{H^{1-\sigma}}^{1/2}\|w\|_{H^{1+\sigma}},
\end{equation}
where the constant $C$ is independent of $u,v,w$.
\end{lemma}
\begin{proof}[Proof of the Lemma] We prove the inequality for $b(u,v,w)$ only (for the term $b(v,u,w)$ the proof is completely analogous). To this end, we  rewrite:
$$
(u,\Nx)v=\sum_{i=1}^2u_i\partial_{x_i}v=
\sum_{i=1}^2\partial_{x_i}(u_iv)-v\divv u=\sum_{i=1}^2\partial_{x_i} (u_iv)
$$
and therefore, after the integration by parts, we have
$$
b(u,v,w) =-\sum_{i=1}^2(u_i,v.\partial_{x_i}w).
$$
Then, the H\"older inequality gives
$$
|b(u,v,w)|\le \|u\|_{L^4(\Omega)}\|v\|_{L^p(\Omega)}\|\Nx w\|_{L^q(\Omega)},\ \ \frac1p=\frac14+\frac\sigma2,\ \ \frac1q=\frac12-\frac\sigma2.
$$
Using the Sobolev embedding together with the description of fractional powers of the Stokes operator, we have
$$
\|\Nx w\|_{L^q}\le C\|w\|_{H^{1+\sigma}}.
$$
Analogously, using the Sobolev embedding together with the interpolation, we end up with
$$
\|v\|_{L^p}\le C\|v\|_{H^{1/2-\sigma}}\le C\|v\|_{H^{-\sigma}}^{1/2}\|v\|_{H^{1-\sigma}}^{1/2}
$$
and finish the proof of the lemma.
\end{proof}
We now return to equation \eqref{4.NS-dif}. We first derive the estimate in the $H^{-\sigma}$-norm. To this end we multiply the equation by $A^{-\sigma}\tilde u$ for some $\sigma<1/2$ and get
\begin{multline}\label{4.gron-4}
\frac12\frac d{dt}\|\tilde u\|_{H^{-\sigma}}^2+\|\tilde u\|^2_{H^{1-\sigma}}+b(u_1,\tilde u,A^{-\sigma}\tilde u)+b(\tilde u,u_2,A^{-\sigma}\tilde u)=\\=(\Pr(\eb)\bar \nu,A^{-\sigma}\tilde u)-b(u_1,v,A^{-\sigma}\tilde u)-b(v,u_2,A^{-\sigma}\tilde u).
\end{multline}
Using now estimate \eqref{4.in} together with the fact $\|A^{-\sigma}\tilde u\|_{H^{1+\sigma}}=\|\tilde u\|_{H^{1-\sigma}}$ and the dissipative estimate for $u_1$ and $u_2$, we arrive at
\begin{multline}\label{4.gron-5}
\frac d{dt}\|\tilde u(t)\|^2_{H^{-\sigma}}+
\|\tilde u(t)\|^2_{H^{1-\sigma}}-C\(\|u_1(t)\|_{L^4}^4+\|u_2(t)\|^4_{L^4}\)\|\tilde u(t)\|^2_{H^{-\sigma}}\le\\\le C\|\Pr(\eb)\bar \nu(t)\|_{V'}^2+C\|v(t)\|_{L^4}^2\(\|u_1(t)\|_{H^1}+\|u_2(t)\|_{H^1}\),
\end{multline}
where the constant $C$ is independent of $\eb$, $u_1$, $u_2$, $\nu_1$, $\nu_2$ and $t$. We estimate the $L^4$-norms using the standard interpolation inequality
$$
\|u\|_{L^4(L^4)}^4\le C\|u\|_{L^\infty(L^2)}^2\|u\|_{L^2(H^1)}^2.
$$
Then, applying the Gronwall inequality to \eqref{4.gron-5} and using the dissipative estimate \eqref{4.dis} for $u_1$ and $u_2$ together with estimate \eqref{4.lin-small} for the function $v$, we arrive at
\begin{equation}\label{4.negative}
\|\tilde u(t)\|_{H^{-\sigma}}^2+\int_t^{t+1}\|\tilde u(s)\|^2_{H^{1-\sigma}}\,ds\le Ce^{Kt}\(\|\tilde u(0)\|^2_{H^{-\sigma}}+\|\Pr(\eb)\bar \nu\|_{L_b^2(0,t;V')}^2+\eb\),
\end{equation}
where $C$ and $K$ are independent of $u_1$, $u_2$, $\nu_1$, $\nu_2$ and $t$.
\par
At the final step, we write out estimate \eqref{4.gron-5} with $\sigma=0$ and multiply it on $t$ in order to get
\begin{multline}\label{4.gron-6}
\frac d{dt}(t\|\tilde u(t)\|^2_{H})
-C\(\|u_1(t)\|_{L^4}^4+\|u_2(t)\|^4_{L^4}\)(t\|\tilde u(t)\|^2_{H})\le\\\le Ct\|\Pr(\eb)\bar \nu(t)\|_{V'}^2+Ct\|v(t)\|_{L^4}^2\(\|u_1(t)\|_{H^1}+\|u_2(t)\|_{H^1}\)+\|\tilde u(t)\|_H^{2},
\end{multline}
apply the Gronwall inequality again and use \eqref{4.negative} in order to estimate the last term in the RHS:
$$
\int_0^t\|\tilde u(s)\|^2_H\,ds\le C(t+1)\int_t^{t+1}\|\tilde u(s)\|^2_{H^{1-\sigma}}\,ds.
$$
This gives
\begin{equation}\label{4.smo}
\|\tilde u(t)\|_{H}^2 \le C\frac{t+1}t e^{Kt}\(\|\tilde u(0)\|^2_{H^{-\sigma}}+\|\Pr(\eb)\bar \nu\|_{L^2(0,t;V')}^2+\eb\).
\end{equation}
Up to scaling $\eb\to\eb^2$ in the definition of $\Pr(\eb)$, the last estimates and together with estimates \eqref{4.lin-small} gives us the desired estimate \eqref{1.squeez} (where the space $M_b$ is replaced by $L^2_b$, see Remark \ref{Rem1.NS}). Thus, the assumptions of the abstract Theorem \ref{Th1.main-ent} are verified and this finishes the proof of Theorem \ref{Th4.main-ent}.
\end{proof}

\section{Examples and concluding remarks}\label{s7}

In this concluding section, we illustrate the obtained results on two concrete examples of non-translation compact external forces which lead to finite-dimensional uniform attractors and discuss some related things and open problems.

\begin{example}\label{Ex6.1} Let $e_1=e_1(x)$ be the first eigenvector of the Dirichlet-Laplacian in $L^2(\Omega)$ and
\begin{equation}\label{6.time}
\mu(t):=\(2\cos(t^2)-\frac{\sin(t^2)}{t^2}\)e_1.
\end{equation}
This external force  obviously belongs to the space $M_b^{una}(\R,H)\cap M_b^{s-reg}(\R,H)$ (since it is finite-dimensional and belongs to $L^\infty(\R,H)$), but does not belong to $M_b^{t-reg}(\R,H)$ (since it is not uniformly continuous in time). Thus, $\mu$ is not translation-compact and the classical theory is not applicable for estimating the entropy of the corresponding uniform attractor.  However, $\mu$ is weakly-regular, so the theory developed above works and we may use Theorem \ref{Th2.main-ent} and estimate \eqref{2.ent-1} in order to estimate the Kolmogorov's entropy of the uniform attractor $\Cal A_{un}$ of damped wave equation \eqref{2.wave} with such an external force.
\par
To this end, we first need to specify $\delta=\delta(\eb)$ in the definition of the projector operator $\Pr(\eb)=S_{\delta(\eb)}\circ P_{N(\eb)}$. In particular, since $\mu$ is finite-dimensional in space (actually one-dimensional), we need not to use the projector $P_{N(\eb)}$ and $\Pr(\eb)=S_{\delta(\eb)}$. We recall that this $\delta$ is actually specified in Lemma \ref{Lem2.key}, so we only need to choose $\delta=\delta(\eb)$ in such a way that the solution of the ODE
$$
v''(t)+\gamma v'(t)+\lambda_1 v(t)+v(t)=(1-S_\delta)\mu_1(t),\ \ \mu_1:=(\mu,e_1)
$$
will be of order $\eb$ in the standard energy norm. Furthermore, according to \eqref{2.mest}, we see that this norm is controlled by
$$
C\delta+\sup_{|t-s|\le\delta}\big|\int_s^t\mu_1(s)\,ds\big|\le C_1\delta
$$
since $\mu\in L^\infty(\R,H)$. Thus, we may take
\begin{equation}\label{6.delta}
\delta(\eb)=a\eb,
\end{equation}
where the positive constant $a$ is independent of $\eb$.
\par
At the next step, we find explicitly the hull $\Cal H(\mu)$. To this end, it is enough to note that $T(h)\mu\to 0$ as $h\to\pm\infty$ weakly-star in $L^\infty_{loc}(\R,H)$ and, therefore,
$$
\Cal H(\mu)=\{T(h)\mu,\ h\in\R\}\cup\{0\}.
$$
Indeed, for every $\psi\in C_0^\infty(0,1)$, we have
\begin{multline*}
|\int_0^{1}(T(h)\mu(t),\psi(t))dt|=\int_0^{1}(\frac{\sin((t+h)^{2})}{t+h})'\psi(t)dt=\\=
|-\int_0^{1}(\frac{\sin((t+h)^{2})}{t+h})\psi'(t)dt|\le  C\|\frac{\sin((t+h)^{2})}{t+h})\|_{L^{\infty}(0,1)}\rightarrow0 \ \text{as} \ \ h\rightarrow\pm\infty.
\end{multline*}
Thus,
$$
S_\delta\Cal H(\mu)=\left\{e_1(x)\delta^{-1}
\(\frac{\sin((t+h+\delta)^2)}{t+h+\delta}-\frac{\sin((t+h)^2)}{t+h}\),\ \ h\in\R\right\}\cup\{0\}
$$
and we only need to estimate the entropy of $S_\delta\Cal H(\mu)\big|_{t\in[0,m]}$. To do so, we note that the function $|T(h)f_\delta(t)|$ is less than $\eb/2$ in $L^\infty(\R)$-norm
for all $t\in[0,m]$ if
$$
h\notin[-m-4\delta^{-1}\eb^{-1},4\delta^{-1}\eb^{-1}],\ \ f_\delta(t):=\delta^{-1}\(\frac{\sin((t+\delta)^2)}{t+\delta}-\frac{\sin(t^2)}{t}\).
$$
 Indeed,
$$
|T(h)f_\delta(t)|=\delta^{-1}\big|\frac{\sin((t+h+\delta)^2)}{t+h+\delta}-\frac{\sin((t+h)^2)}
{t+h}\big|\le\delta^{-1}(\frac1{|t+h+\delta|}+\frac1{|t+h|})\le\frac{2}{\delta|t+h|}<\frac{\eb}{2}
$$
and we need
$$
|t+h|>\frac{4}{\delta\eb}
$$
for $t\in[0,m]$, so it is enough to take
$$
h>4\delta^{-1}\eb^{-1} \ \ \text{or}\ \ h<-m-4\delta^{-1}\eb^{-1}.
$$
Thus, it is enough to find the entropy of a curve $h\to T(h)f_\delta$, $h\in   [-m-4\delta^{-1}\eb^{-1},4\delta^{-1}\eb^{-1}]$ in the space $L^\infty(0,m)$. Note that $f'_\delta(t)$ is uniformly bounded in $L^\infty$ with respect to $\delta\to0$:
$$
|f'_\delta(t)|=\bigg|\delta^{-1}[2\cos((t+\delta)^2)-
\frac{\sin((t+\delta)^2)}{(t+\delta)^2}-2\cos(t^2)+\frac{\sin(t^2)}{t^2}]\bigg|\le 6\delta^{-1}.
$$
Then this curve is Lipschitz continuous with the Lipschitz constant $6\delta^{-1}$. This gives the desired estimate.
$$
\Bbb H_{\eb}(S_\delta\Cal H(\mu)\big|_{[0,m]},M_b([0,m],H))\le \ln\(6\delta^{-1}{\eb^{-1}}(m+8\delta^{-1}\eb^{-1})\).
$$
Inserting $\delta=a\eb$ and $m=K\log_2\eb^{-1}$, we get
$$
\Bbb H_{\eb/K}(S_{\delta(\eb)}\Cal H(\mu)\big|_{[0,L\log_2(\frac{2R_0}\eb)]}, M_b([0,L\log_2(\frac{2R_0}\eb)],H))\le 4\log_2(\frac {2R_0}\eb)+C.
$$
Finally, Theorem \ref{Th2.main-ent} gives us the estimate
\begin{equation}
\Bbb H_{\eb}(\Cal A_{un},E)\le (D+4)\log_2(\frac {2R_0}\eb)+C.
\end{equation}
Thus, we have proved that the uniform attractor $\Cal A_{un}$ of wave equation \eqref{2.wave} with the right-hand side $\mu$ defined by \eqref{6.time} possesses a uniform attractor with finite fractal dimension.
\par
We also note that this result is not restricted to the subcritical wave equation only. Similar estimate also holds in the critical case with periodic boundary conditions if we replace $e_1(x)$ by $1$. Moreover, if we replace $e_1$ by the first eigenvector of the Stokes operator, the corresponding uniform attractor of Navier-Stokes system \eqref{4.NS} will be also finite-dimensional. The proof of these facts is completely analogous, so we leave it to the reader.
\end{example}
\begin{example}\label{Ex6.2} We now consider time-regular analogue of \eqref{6.time}. Let $\{e_n\}_{n=1}^\infty$ be the orthonormal in $H=L^2(\Omega)$ base of eigenvectors of the Dirichlet-Laplacian and let
\begin{equation}\label{6.space}
\mu(t)=(t-n+1)e_n+(n-t)e_{n+1},\ \ t\in[n-1,n],\ n\in\Bbb N,\ \ \mu(-t)=\mu(t).
\end{equation}
Then, $\mu\in W^{1,\infty}(\R,H)$ with the norm $\le 2$ and therefore it is time-regular, however, it is obviously not space-regular and therefore is not translation compact.
\par
 Let us apply Theorem \ref{Th2.main-ent} to the wave equation \eqref{2.wave} with the right-hand side \eqref{6.space}. To this end, as in the previous example, we first need to specify the projector $\Pr(\eb)=P_{N(\eb)}$, i.e. to specify the function $N(\eb)$ (since $\mu$ is time-regular, we need not the projector $S_{\delta}$). To this end, we return to Lemma \ref{Lem2.key} and its proof again. Since $\mu\in W^{1,\infty}(\R,H)$, we have $\nu=h_{time}$ and the solution of
 $$
 \Dt^2 \hat v_2+\gamma\Dt\hat v_2-\Dx\hat v_2+\hat v_2=Q_N\nu
 $$
 satisfies $\|\xi_{\hat v_2}(t)\|_{E^1}\le C\|\nu\|_{W^{1,\infty}(\R,H)}\le C_1$.
 \par
 Using now the Weyl asymptotic $\lambda_n\sim Cn^{2/3}$ we infer that
 \begin{equation}
 \|Q_N\xi_{\hat v_2}(t)\|_{E}+\|Q_N\hat v_2(t)\|_{L^{12}}\le C\lambda_N^{-3/8}
 \|\xi_{\hat v_2}(t)\|_{E^{1}}\le CN^{-1/4}.
\end{equation}
This shows that
$$
N(\eb)\sim a\eb^{-4}.
$$
We also note that, as in the previous example, $T(h)\mu\to0$ as $h\to\pm\infty$ weakly-star  in $M_{loc}(\R,H)$, therefore
$$
\Cal H(\mu)=\{T(h)\mu, \ h\in\R\}\cup\{0\}.
$$
It remains to find the entropy of the restrictions $P_{N(\eb)}\Cal H(\mu)\big|_{t\in[0,m]}$ . To do this we note that, by the definition of $\mu$,
$$
P_{N(\eb)}\mu(t)\equiv 0 \text{ if } \ t\notin [-N(\eb)-1,N(\eb)+1].
$$
Thus, we only need to estimate the entropy of the curve $h\to T(h)\mu$, $h\in[-N(\eb)-1,N(\eb)+1+m]$. Since this curve is Lipschitz continuous with Lipschitz constant $2$,
$$
\Bbb H_\eb(P_{N(\eb)}\Cal H(\mu)\big|_{[0,m]},M_b([0,m],H))\le \log_2\(\eb^{-1}(2N(\eb)+2+m)\)
$$
and
$$
\Bbb H_{\eb/K}(P_{N(\eb)}\Cal H(\mu)\big|_{[0,L\log_2(\frac{2R_0}\eb)]},M_b([0,L\log_2(\frac{2R_0}\eb)],H))\le 6\log_2(\frac{2R_0}\eb)+C
$$
and formula \eqref{2.ent-1} gives us
$$
\Bbb H_\eb(\Cal A_{un},H)\le (D+6)\log_2(\frac{2R_0}\eb)+C.
$$
Thus, in the case of time-regular external forces $\mu$ given by \eqref{6.space}, the uniform attractor $\Cal A_{un}$ also has finite fractal dimension.
\end{example}
\begin{remark} The structure of the classes of external forces introduced above deserves further attention. Indeed, as shown in  Proposition \ref{Prop0.str}, the space of weakly regular measures is a closure of the sum of two Banach spaces (the spaces of time and space regular measures respectively, see \eqref{0.main}). It is well-known that an algebraic sum of two closed subspaces is not necessarily closed even in the case of Hilbert spaces, but it would be interesting to give a concrete example of a weakly regular measure which is not a sum of space and time regular measures or to prove that this is impossible.
\par
Similar questions also arise when studying the normal and weakly normal external forces. For instance, is it true that
$$
L^{p,w-norm}_b(\R,H)=[L^{p,norm}_b(\R,H)+ L^{p,s-reg}_b(\R,H)]_{L^p_b(\R,H)}?
$$
or even without the closure? In contrast to the proof of Proposition \ref{Prop0.str}, we unfortunately do not know how to present an "$\eb$-normal" function in a sum of small and normal ones (i.e., what is the analogue of the mollifying operator used there on the level of normal functions).
\par
Such structural theorems for various  classes of external forces have not only theoretical interest, but may also help to invent new more effective ways to verify the compactness of uniform attractors and estimating their entropy.
\end{remark}
\begin{remark} We conclude this section by recalling that, in all results concerning weakly damped wave equations stated above, we require the external measure $\mu$ to be weakly non-atomic ($\mu\in M^{una}_b(\R,H)$). As shown in \cite{SZUMN}, this is in a fact necessary and sufficient condition for the associated extended semigroup $\Bbb S_t$ to be continuous in $\Bbb E^w$. This continuity is in turn used in order to verify the representation formula
\begin{equation}\label{6.str}
\Cal A_{un}=\cup_{\nu\in\Cal H(\mu)}\Cal K_\nu\big|_{t=0}
\end{equation}
and this formula is crucial for our method of estimating the Kolmogorov $\eb$-entropy of the uniform attractor as was already pointed out above. The counterexamples which show that \eqref{6.str} may fail without the assumption on the external measure to be weakly non-atomic are also presented in \cite{SZUMN}. On the other hand, as shown there, the uniform attractor $\Cal A_{un}$ still exists in a strong topology of $E$ even without \eqref{6.str} if the external measure belongs to $M_b(\R,H^1_0(\Omega))$, so we expect this result to be true, e.g. for space-regular external forces.
\par
Thus, an interesting and important open problem arises here: how to estimate the Kolmogorov $\eb$-entropy of uniform attractors which do not possess a representation formula \eqref{6.str}?  Up to the moment, it is not clear how to do this even in the case when $\mu$ is translation-compact in $M_b(\R,H)$.
\end{remark}

\end{document}